\documentclass[hidelinks, letterpaper]{siamart1116}
\newcommand{\C}{\mathbb{C}}
\usepackage{jeff}
\usepackage{xfrac}
\renewcommand{\E}{\mathbb{E}}

\usepgfplotslibrary{groupplots}
\usepgfplotslibrary{fillbetween}

\DeclareMathOperator{\Range}{Range}

\usepackage{booktabs}

\usepackage[algo2e, lined, vlined, linesnumbered, english]{algorithm2e}

\SetCommentSty{mycommfont}
\SetKwComment{Comment}{}{}
\SetKwInOut{Input}{Input}
\SetKwInOut{Output}{Output}
\SetSideCommentRight	% Push comments to far right
\SetKwComment{comment}{}{}

\SetAlCapSkip{1pc}

\usepackage{enumitem}

\newcommand{\TheTitle}{Projected Nonlinear Least Squares for Exponential Fitting}
\newcommand{\TheAuthors}{Jeffrey M. Hokanson}
\headers{\TheTitle}{\TheAuthors}

\DeclareMathOperator{\expm}{{\tt expm1}}
\DeclareMathOperator{\expdiff}{{\tt expdiff}}
\newcommand{\eff}{\eta}

\title{Projected Nonlinear Least Squares \\ for Exponential Fitting\thanks{Submitted to the editors 14 July 2016.
\funding{This work was supported by NSF VIGRE grants DMS-0240058 and DMS-0739420
	at Rice University and the Department of Defense, Defense Advanced Research
	Project Agency’s program Enabling Quantification of Uncertainty in Physical Systems.
}}}

\author{
  Jeffrey M. Hokanson\thanks{
	Department of Applied Mathematics and Statistics, Colorado School of Mines,
	1500 Illinois St, Golden CO, 80401,
	(\email{hokanson@mines.edu}, \url{http://inside.mines.edu/\string~hokanson/}).}
}

\begin{document}
\maketitle
\begin{abstract}
	The modern ability to collect vast quantities of data presents a challenge for parameter estimation problems.
	Posed as a nonlinear least squares problem fitting a model to the data,
	the cost of each iteration grows linearly with the amount of data;
	with large data, it can easily become too expensive to perform many iterations.
	Here we develop an approach that projects the data onto a low-dimensional subspace 
	that preserves the quality of the resulting parameter estimates.
	We provide results from both an optimization and a statistical perspective
	that shows that accurate parameter estimates are obtained when the subspace angles
	between this projection and the Jacobian of the model at the current iterate remain small.
	However, for this approach to reduce computational complexity, 
	both the projected model and projected Jacobian must be computed inexpensively.
	This places a constraint on the pairs of models and subspaces for which this approach provides a computational speedup.
	Here we consider the exponential fitting problem projected onto the range of a Vandermonde matrix,
	for which the projected model and projected Jacobian can be computed in closed form using a generalized geometric sum formula.
	We further provide an inexpensive heuristic that picks this Vandermonde matrix
	so that the subspace angles with the Jacobian remain small
	and use this heuristic to update the subspace during optimization.
	Although the asymptotic cost still depends on the data dimension,
	the overall cost of solving this sequence of projected nonlinear least squares problems
	is less expensive than the original.
	Applied to the exponential fitting problem,
	this yields an algorithm that is not only faster in the limit of large data
	than the conventional nonlinear least squares approach,
	but is also faster than subspace based approaches such as HSVD.
\end{abstract}



\section{Introduction}
With the increasing prowess of data acquisition hardware and storage,
collecting vast amounts of data has become trivial.
This poses a challenge for parameter estimation problems
where the sheer scale of data makes these problems expensive.
Here we consider a nonlinear least squares parameter estimation problem~\cite{HPS13} that seeks
to fit a model $\ve f$ with $q$ parameters $\ve \theta \in \C^q$
to (noisy) measurements $\tve y$ yielding a (noisy) parameter estimate $\tve\theta$
that minimizes the $2$-norm mismatch
\begin{equation}\label{eq:nls}
	\tve \theta := \argmin_{\ve\theta\in \C^q} \| \ve f(\ve\theta) - \tve y\|_2^2, \quad \text{where} \quad 
	 \ve f:\C^q\to \C^n, \quad \tve y\in \C^n, \quad q \ll n.
\end{equation}
With vast quantities of data, 
the asymptotic cost of solving this problem is dominated by the $n$-dependent steps in the optimization.
For example, using either Gauss-Newton or Levenberg-Marquardt,
each optimization step solves a least squares problem involving the Jacobian of $\ve f$,
$\ma J:\C^q \to \C^{n\times q}$, at a cost of $\order(nq^2)$ operations~\cite[Ch.~9]{Bjo96}.
To reduce this cost, we propose replacing the full least squares problem~\cref{eq:nls}
with a sequence of low-dimensional surrogate problems by projecting measurements onto $\tve y$
onto a sequence of subspaces $\set W_\ell\subset \C^n$ with $m_\ell:= \dim \set W_\ell \ll n$
\begin{equation}\label{eq:pnls}
	\tve\theta_{\set W_\ell} := \argmin_{\ve\theta \in \C^q} \| \ma P_{\set W_\ell} [ \ve f(\ve\theta) - \tve y\, ]\|_2^2
	= \argmin_{\ve\theta \in \C^q} \| \ma W_\ell^*\ve f(\ve\theta) - \ma W_\ell^*\tve y\|_2^2,
	\quad \ma P_{\set W_\ell} = \ma W_\ell^{\phantom{*}\!} \ma W_\ell^*,\!\!
\end{equation}
where $\ma P_{\set W_\ell}$ is an orthogonal projector onto $\set W_{\ell}$
and $\ma W_\ell \in \C^{n\times m_\ell}$ is an orthonormal basis for $\set W_\ell$.
Although the total cost is still asymptotically $n$-dependent due to the multiplication $\ma W_\ell^*\tve y$,
each optimization step is cheaper since the projected Jacobian $\ma W^*_\ell\ma J(\ve\theta)$ is smaller.
However, for this computational speedup to be fully realized, 
the products $\ma W_\ell^*\ve f(\ve\theta)$ and $\ma W_\ell^*\ma J(\ve\theta)$
need to be formed without the expensive, $n$-dependent multiplication.
Additionally, we must ensure that the final projected parameter estimate $\tve\theta_{\set W_\ell}$
remains a good estimate of the full parameter estimate, $\tve\theta$.
Here we do so by requiring that the subspace angles between $\set W_\ell$ 
and the Jacobian at the current iterate remain small.
This requirement is justified by perspectives from both optimization and statistics.
From an optimization perspective described in \cref{sec:optimization},
the accuracy of each optimization step depends on these subspace angles
and the projected parameter estimate $\tve\theta_{\set W_\ell}$ is 
equal to the full parameter estimate $\tve\theta$ when
these subspace angles go to zero.
From a statistical perspective described in \cref{sec:statistical},
when measurements $\tve y$ are contaminated by additive Gaussian noise,
the covariance of projected parameter estimate $\tve\theta_{\set W_\ell}$ 
is larger than the covariance of full parameter estimate $\tve\theta$
by an amount that scales with these subspace angles
as measured by \emph{efficiency}.
Hence the subspace angles between $\set W_\ell$ and the Jacobian at the current iterate
determine the quality of our projected parameter estimate $\tve\theta_{\set W_\ell}$.
The challenge in applying this projected nonlinear least squares approach 
to a specific problem is satisfying both criteria simultaneously:
finding a sequence of subspaces $\lbrace \set W_\ell \rbrace_\ell$ with orthogonal bases $\lbrace \ma W_\ell \rbrace_\ell$ where 
$\ma W^*_\ell\ve f(\ve\theta)$ and $\ma W^*_\ell\ma J(\ve\theta)$ can be formed inexpensively independently of $n$
and where the subspace angles between $\set W_\ell$ and the range of $\ma J(\ve\theta_k)$ for each iterate $\ve\theta_k$ remain small.

Here we consider the \emph{exponential fitting problem}~\cite{PS10},
also known as \emph{modal analysis}~\cite{Ewi84}, \emph{harmonic estimation}~\cite{JS11},
and \emph{spectral analysis}~\cite{SM97},
that seeks to approximate data $\tve y$ as a sum of $p$ complex exponentials
with frequencies $\ve\omega$ and amplitudes $\ve a$ where 
\begin{equation}
	[\ve f([\ve\omega, \ve a])]_j = \sum_{k=1}^{p} a_k e^{j \omega_k}, 
		\quad \ve \omega, \ve a\in\C^p; \quad \ve\theta = [\ve\omega, \ve a],  \quad q = 2p.
\end{equation}
There is an extensive body of literature on this problem,
with a wide array of methods for recovering the frequencies $\ve\omega$:
from classical approaches such as Prony's method~\cite{Pro95}
%reviewed in~\cite{PT14}, % This review mainly focuses on the connection of Prony to structured function approximation
its extensions for overdetermined problems~\cite[\S9.4]{Hil56},
to subspace methods~\cite{HK66} such as HSVD~\cite{BBO87}, HTLS~\cite{HCDH94}, 
and the matrix-pencil method~\cite{HS90},
to parameter estimation approaches using optimization (such as ours)~\cite{HPS13,VBH97},
to more recent approaches based on ideas from sparse recovery~\cite{TBSR13},
and many others described in reviews~\cite{IV99,KM78,PS10,VSHH01}.
We choose this problem due to the exploitable structure of the model function $\ve f([\ve\omega, \ve a])$
that allows us to obtain inexpensive inner products with subspaces that approximately contain the range of the Jacobian.
Specifically, as the model function is the product of a Vandermonde matrix $\ma V(\ve\omega)$
and the amplitudes~$\ve a$,
\begin{equation}
	\ve f([\ve\omega, \ve a]) = \ma V(\ve\omega) \ve a, \qquad [\ma V(\ve\omega)]_{j,k} = e^{j \omega_k},
\end{equation}
by projecting measurements onto the subspace $\set W(\ve\mu)$,
\begin{equation}
	\set W(\ve\mu) := \Range \ma V(\ve\mu) = \Range \ma W(\ve\mu),
	\qquad \ma W(\ve\mu)^*\ma W(\ve\mu) = \ma I, \quad
	\ve\mu \in \C^m,
\end{equation}
we can inexpensively obtain the inner products $\ma W(\ve\mu)^*\ve f([\ve\omega,\ve a])$
and $\ma W(\ve \mu)^*\ma J([\ve\omega, \ve a])$ as described in \cref{sec:closed}
using the geometric sum formula and its generalization given in \cref{sec:geometric}.
Further, using a heuristic described in~\cref{sec:subspace},
we can ensure the subspace angles between $\set W(\ve\mu)$ and $\Range \ma J([\ve\omega, \ve a])$
remain small by a careful selection of $\ve\mu$.
The net result is a faster solution to the exponential fitting problem
in the limit of large data as illustrated by a magnetic resonance spectroscopy test case in \cref{sec:example}.
Further, due to careful selection of the subspaces,
the projected parameter estimate $\tve\theta_{\set W}$ remains close to the full parameter estimate $\tve\theta$
as seen in \cref{fig:dispersion} for a toy problem and in \cref{fig:nmr_pert} for the magnetic resonance spectroscopy test case.

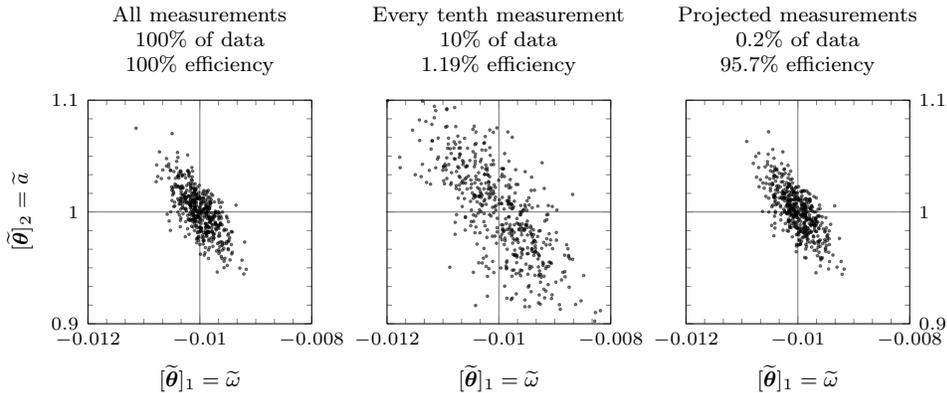
\begin{figure}
	\centering
		\begin{tikzpicture}
			\begin{groupplot}[group style={group size=3 by 1},
				xmin = -0.012, xmax = -0.008,
				xtick = {-0.012, -0.01, -0.008},
				ymin = 0.9, ymax = 1.1,
				ytick = {0.9, 1, 1.1},
				width = 0.35\textwidth, height = 0.35\textwidth,
				xlabel = {$[\tve\theta]_1 = \widetilde\omega$},
				ylabel = {$[\tve\theta]_2 = \widetilde a$},
				xticklabel style={/pgf/number format/fixed, /pgf/number format/precision=3},
				minor tick num=5,
				title style={align=center},
				scaled x ticks = false
				]
				\nextgroupplot[title={All measurements \\ $100\%$ of data\\ $100\%$  efficiency} ]
				\addplot[black, only marks, mark size = 0.5pt, opacity=0.5, mark options = {line width=0pt}] 
					table {fig_dispersion_full.dat};
				\addplot[gray] coordinates {(-0.01, 0) (-0.01, 2)};
				\addplot[gray] coordinates {(-0.1, 1) (0, 1)};
				\nextgroupplot[title={Every tenth measurement  \\ $10\%$ of data\\ $1.19\%$ efficiency}
					, ylabel = {}, yticklabels = {}]
				\addplot[black, only marks, mark size = 0.5pt, opacity=0.5, mark options = {line width=0pt}] 
					table {fig_dispersion_subsample.dat};
				\addplot[gray] coordinates {(-0.01, 0) (-0.01, 2)};
				\addplot[gray] coordinates {(-0.1, 1) (0, 1)};
				\nextgroupplot[title={Projected measurements \\ $0.2\%$ of data\\ $95.7\%$ efficiency},
							 yticklabel pos=right, ylabel = {}]
				\addplot[black, only marks, mark size = 0.5pt, opacity=0.5, mark options = {line width=0pt}] 
					table {fig_dispersion_reduced_opt.dat};
				\addplot[gray] coordinates {(-0.01, 0) (-0.01, 2)};
				\addplot[gray] coordinates {(-0.1, 1) (0, 1)};
			\end{groupplot}
		\end{tikzpicture}
	\caption{
		Parameter estimates from a toy exponential fitting problem with true parameters
		$\widehat{\omega} = -0.01$ and $\widehat{a} = 1$ using $n=1000$ measurements contaminated 
		with zero-mean Gaussian noise $\ve g$ with $\Cov\ve g = 0.01\ma I$,
		$\tve y = \ve f([-0.01, 1]) + \ve g$.
		The parameter estimates on the left are computed by solving \cref{eq:nls}
		with all measurements $\tve y$.
		In the center, the parameter estimates are computed by solving \cref{eq:pnls}
		using a subspace that selects every 10th measurement.
		On the right, the parameter estimates are also computed by solving \cref{eq:pnls},
		but instead using the subspace $\set W(\ve\mu)$ where $\ve\mu = [-0.008\pm 0.0014i]$.
		In each case, the resulting nonlinear least squares problem was solved using Matlab's {\tt lsqnonlin}.
		Efficiency, defined in \cref{sec:statistical}, quantifies how close the 
		covariance of the projected problem resembles the full problem (left).
		As this example shows, the projected parameter estimate with well-chosen subspace yields 
		almost identical parameter estimates to the full problem. 
		}
	\label{fig:dispersion}
\end{figure}

Projection is a recurring theme in applied mathematics, appearing in a variety of contexts
from Galerkin projections for solving partial differential equations
to the randomized projections that form the foundation of randomized numerical linear algebra.
Our projection approach for solving a nonlinear least squares problem
fits into this theme and it is not without precedent.
Incremental methods, such as \emph{incremental gradient}~\cite{Ber97} and the \emph{extended Kalman filter}~\cite{Ber96},
when applied to a nonlinear least squares problem can be interpreted as projecting onto a row (or set of rows) at each iteration,
choosing the basis $\ma W_\ell = [\ma I]_{\cdot,\set I_\ell}$ where $\set I_\ell$ the set of rows at the $\ell$th step.
With this perspective, we note that both our method and incremental methods 
require the projected model $\ma W_\ell^*\ve f(\ve\theta)$ and projected Jacobian $\ma W_\ell^*\ma J(\ve\theta)$ to be formed inexpensively.
Satisfying this requirement is straightforward for incremental methods when $\ve f(\ve\theta)$ is defined entry-wise,
whereas in our case we must be careful choose the orthonormal basis $\ma W_\ell$ such that these products can be, for example, evaluated in closed-form.
However these methods differ is in how the basis $\ma W_\ell$ is chosen.
For incremental methods, the set of rows is typically chosen either deterministically by cycling through rows or by randomly selecting rows~\cite{FS12};
whereas in our case, we carefully choose the basis $\ma W_\ell$ such that our steps are accurate,
and, when the data is contaminated by noise, our parameter estimates are precise.

\section{An optimization perspective\label{sec:optimization}}
In this section we provide three different results that inform 
the choice of subspace $\set W_\ell$ from the perspective of optimization.
Each of these results points to the key role played by the canonical subspace angles between
the subspace $\set W_\ell$ and the range of the Jacobian at the current iterate.
We define these canonical subspace angles following Bj\"orck and Golub~\cite[Thm.~1]{BG73}:
if $\set A$ and $\set B$ are two subspaces of $\C^n$ and if $\ma A\in \C^{n\times m_a}$ and $\ma B\in \C^{n \times m_b}$
are orthonormal bases for $\set A$ and $\set B$, 
then the canonical subspace angles $\phi_k(\set A, \set B)$ between $\set A$ and $\set B$ are
\begin{equation}\label{eq:subspace_angles}
	\!\cos \phi_k(\set A, \set B) \!:= \! \sigma_k(\ma A^*\ma B), \ \ 
	0\! \le\! \phi_1(\set A, \set B) \!\le\! \phi_2(\set A, \set B) \!\le\!\! \cdots \!\le\! \phi_{\min\lbrace m_a, m_b\rbrace} (\set A, \set B)\! \le\! \pi/2\!\!\!
\end{equation}
where $\sigma_k(\ma X)$ is the $k$th singular value of $\ma X$ in descending order.
Our first result in \cref{sec:optimization:first} uses the first order necessary conditions to observe that 
the projected problem will have the same stationary points as the full problem
when the subspace angles between the $\set W$ and the range of the Jacobian at the stationary point are zero.
Our second result in \cref{sec:optimization:proximity} shows that the difference between the Gauss-Newton steps
of the full and projected problems depends on the subspace angles between $\set W_\ell$
and the range of the Jacobian at the current iterate.
Our third result in \cref{sec:optimization:inexact} interprets the Levenberg-Marquardt method applied to the projected problem
as computing inexact steps of the Levenberg-Marquardt method applied to the full problem.
We show that by making the subspace angles between $\set W_\ell$ and the range of the Jacobian at the current iterate
small, we can satisfy one of the conditions for the convergence of inexact Newton.
All these results suggest that the subspace angles between $\set W_\ell$
and the range of the Jacobian at the current iterate should be small.

\subsection{First order optimality\label{sec:optimization:first}}
The first order necessary conditions for a point $\cve \theta$ to be a local optimum
require that the gradient of the objective function at this point be zero~\cite[Thm.~2.2]{NW06}.
In the context of nonlinear least squares, where the gradient of the full problem~\cref{eq:nls} is
\begin{equation}
	\nabla_\ve\theta \ \|\ve f(\ve\theta) -\tve y\|_2^2 = 
		2\,\ma J(\ve\theta)^*\ve r(\ve\theta), \qquad \ve r(\ve\theta) :=\ve f(\ve\theta) - \tve y, 
		\quad [\ma J(\ve\theta)]_{\cdot,k} := \frac{\partial \ve r(\ve\theta)}{\partial [\ve\theta]_k}
\end{equation}
a point $\cve\theta$ satisfies the first order necessary conditions if~\cite[\S9.1.2]{Bjo96}
\begin{equation}\label{eq:first_full}
	\ma J(\cve \theta)^* \ve r(\cve \theta) = \ve 0.
\end{equation}
Similarly for the projected problem~\cref{eq:pnls},
a point $\cve\theta_\set W$ will satisfy the first order necessary conditions 
for the projected problem if
\begin{equation}
	\ma J(\cve\theta_\set W)^*\ma P_\set W \ve r(\cve\theta_\set W) = \ve 0.
\end{equation}
To assess the quality of the projected problem, we ask:
under what conditions will $\cve\theta_\set W$ 
also satisfy the first order necessary conditions for the full problem~\cref{eq:first_full}?
There are two conditions under which this happens.
The zero-residual case, where $\ve r(\cve\theta_\set W) = \ve 0$,
implies that measurements $\tve y$ exactly fit the model $\ve f$.
This situation makes the problem easy to solve,
as any subspace $\set W$ that yields a well-posed optimization problem can be used.
The other, more general situation occurs when $\set W$ contains the range of the Jacobian,
as then $\ma P_\set W \ma J(\cve\theta_\set W) = \ma J(\cve\theta_\set W)$.
This is equivalent to requiring all the subspace angles between $\set W$
and $\Range \ma J(\cve\theta_\set W)$ to be zero.
The challenge with this condition is that it is black or white: 
either the $\set W$ contains the range of the Jacobian or it does not.
In the next two subsections we suggest other requirements on $\set W$
that allow more shades of grey.

\subsection{Proximity of steps\label{sec:optimization:proximity}}
Another result that provides insight into the choice of subspace $\set W$
comes from considering the Gauss-Newton step~\cite[\S10.3]{NW06} for the full and projected problems
at $\ve\theta$:
\begin{equation}\label{eq:gn_steps}
	\ve s = -\ma J(\ve\theta)^+\ve r(\ve\theta) \quad \text{(full)}, \qquad 
	\ve s_\set W = -[\ma P_\set W \ma J(\ve\theta)]^+\ma P_\set W \ve r(\ve\theta) \quad \text{(projected)}
\end{equation}
where $\ma A^+$ denotes the pseudoinverse of $\ma A$~\cite[\S5.5.4]{GL96}.
We bound the difference between these two steps using \cref{lem:project_ls} from \cref{sec:bounds}.

\begin{theorem}[Gauss-Newton step accuracy]\label{thm:gn_step}
	Let $\ve s$ and $\ve s_\set W$ be the Gauss-Newton steps for the full and projected problems at $\ve \theta$
	as given in \cref{eq:gn_steps}, then their mismatch is bounded by
	\begin{equation}
		\| \ve s - \ve s_\set W\|_2 \le \| \ma J(\ve\theta)^+\|_2 \| \ve r(\ve\theta)\|_2
			\left[\sin \phi_q(\set W, \set J(\ve\theta) ) + \tan^2\phi_q(\set W, \set J(\ve\theta) ) \right]
	\end{equation}	
	where $\set J(\ve\theta) := \Range \ma J(\ve\theta)$.
\end{theorem}

Using this theorem, we can provide a heuristic for estimating the mismatch between the full and projected parameter estimates.
Applying the Gauss-Newton method to the full problem starting at a stationary point of the projected problem $\cve\theta_\set W$,
we note the first step cannot move further than 
\begin{equation}
		\| \ve s \|_2 \le \| \ma J(\cve\theta_\set W)^+\|_2 \| \ve r(\cve\theta_\set W)\|_2
			\left[\sin \phi_q(\set W, \set J(\cve\theta_\set W) ) + \tan^2\phi_q(\set W, \set J(\cve\theta_\set W) ) \right].
\end{equation}
Although multiple iterations of the Gauss-Newton method might be required to reach a stationary point of the full problem,
if $\cve\theta_\set W$ is sufficiently close to a stationary point $\cve\theta$ of the full problem,
then we expect this first step to yield a good estimate; i.e., $\cve\theta_\set W + \ve s \approx \cve\theta$.
This suggests choosing subspaces $\set W$ to minimize the largest subspace angle between
$\set W$ and the range of Jacobian at the stationary point of the projected problem $\set J(\cve\theta_\set W)$
to ensure the full and projected parameter estimates are nearby.

\subsection{Inexact Levenberg-Marquardt\label{sec:optimization:inexact}}
A third and final result that provides insight into the choice of subspace $\set W$
comes from considering steps of the Levenberg-Marquardt method~\cite[\S10.3]{NW06} 
applied to the projected problem~\cref{eq:pnls} as inexact steps of the Levenberg-Marquardt method applied to
the full problem~\cref{eq:nls}.
For the full problem, the Levenberg-Marquardt method generates a sequence of iterates $\lbrace \ve \theta_k\rbrace_{k}$
starting from a given $\ve\theta_0$ using the rule
\begin{equation}\label{eq:lm_full}
	\ve \theta_{k+1} = \ve\theta_k + \ve s_k, \qquad 
	\ve s_k := \argmin_{\ve s\in \C^q} 
				\left\| \begin{bmatrix} \ma J(\ve\theta_k) \\ \lambda_k \ma I \end{bmatrix} \ve s 
					+ \begin{bmatrix} \ve r(\ve \theta_k) \\ \ve 0 \end{bmatrix} \right\|_2^2
\end{equation}
where $\lambda_k$ has been chosen to enforce a trust region; see, e.g.,~\cite[\S3.3.5]{Kel99}.
Iterates of the projected problem $\lbrace \tve\theta_k \rbrace_{k\ge 0}$
follow a similar update rule:
\begin{equation}\label{eq:lm_projected}
	\tve \theta_{k+1} = \tve\theta_k + \tve s_k, \qquad 
	\tve s_k := \argmin_{\ve s\in \C^q} 
				\left\| \begin{bmatrix} \ma W_k^*\ma J(\ve\theta_k) \\ \lambda_k \ma I \end{bmatrix} \ve s 
					+ \begin{bmatrix} \ma W_k^*\ve r(\ve \theta_k) \\ \ve 0 \end{bmatrix} \right\|_2^2,
\end{equation}
where $\ma W_k$ is the orthonormal basis for the subspace applied at the $k$th step.
Here we ask, under what conditions on $\ma W_k$ does 
the sequence $\lbrace \tve\theta_k\rbrace_k$ converge to
the same point as $\lbrace \ve\theta_k \rbrace_k$?
Although we are unable to prove the convergence of the projected iterates 
unless the subspace angles between $\set J(\tve\theta_k)$ and $\set W_k$ go to zero,
we invoke the convergence analysis of inexact Newton~\cite{DES82},
and specific results for inexact Levenberg-Marquardt~\cite{WH85},
to suggest a choice of $\set W_k$.
These convergence results require that the error in the step $\tve s_k$ 
be bounded by a forcing sequence $\lbrace \alpha_k \rbrace_k$: 
\begin{equation}\label{eq:inexact_error}
	\frac{\| (\ma J(\tve \theta_k)^*\ma J(\tve \theta_k) + \lambda_k^2 \ma I)\tve s_k +
		 \ma J(\tve \theta_k)^*\ve r(\tve \theta_k)\|_2}
		{\|\ma J(\tve \theta_k)^*\ve r(\tve \theta_k)\|_2} \le \alpha_k < \alpha < 1.
\end{equation}
Here we show that the quantity on the left can be bounded above in terms of the 
subspace angles and that this quantity can be made arbitrarily small.

To bound~\cref{eq:inexact_error} we use a result from \cref{sec:bounds} 
applied to the augmented subspace, Jacobian, and residual in the least squares problem for $\tve s_k$, \cref{eq:lm_projected},
\begin{align*}
	\hset W_k := \Range \begin{bmatrix} \ma W_k & \ma 0 \\ \ma 0 & \ma I \end{bmatrix}, \quad
	\hma J_k := \begin{bmatrix} \ma J(\tve\theta_k) \\ \lambda_k \ma I \end{bmatrix}, \quad
	\hve r_k := \begin{bmatrix} \ve r(\ve\theta_k) \\ \ve 0 \end{bmatrix}.
\end{align*}
Then applying \cref{lem:pls_normal} with subspace $\hset W_k$ to \cref{eq:lm_projected} yields,
\begin{equation}\label{eq:lm_inexact_step1}
	\frac{\| (\ma J(\tve \theta_k)^*\ma J(\tve \theta_k) + \lambda_k^2 \ma I)\tve s_k 
		+ \ma J(\tve \theta_k)^*\ve r(\tve \theta_k)\|_2}
		{\|\ma J(\tve \theta_k)^*\ve r(\tve \theta_k)\|_2} 
		\le \frac{\sin \phi_q(\hset W_k, \hset J_k)}{\cos^2 \phi_q(\hset W_k, \hset J_k)}
			\frac{\|\hma J_k\|_2 \| \ma P_{\hset J_k}^\perp \hve r_k \|_2}
				{\|\ma J(\tve\theta_k)^*\ve r(\tve\theta_k)\|_2}.
\end{equation}
where $\hset J_k := \Range \hma J_k$, $\hset R_k := \Range \hve r_k$,
and $\ma P_{\hset J_k}^\perp$ denotes the orthogonal projector onto the subspace perpendicular to $\hset J_k$.
To obtain an expression in terms of subspace angles, we note the numerator can be written in terms of sines
\begin{align*}
	\| \ma P_{\hset J_k}^\perp \hve r_k\|_2 = \sin \phi_1(\hset J_k, \hset R_k) \|\hve r_k\|_2
		= \sin \phi_1(\hset J_k, \hset R_k)\| \ve r(\tve\theta_k)\|_2,
\end{align*}
and similarly we can bound the denominator in terms of subspace angles,
\begin{align*}
	\| \ma J(\tve\theta_k)^*\ve r(\tve\theta_k)\|_2 = \| \hma J_k^*\hve r_k\|_2
	= \| \hma J_k^* \ma P_{\hset J_k} \hve r_k\|_2 \ge 
	\sigma_q(\hma J_k) \cos \phi_1(\hset J_k, \hset R_k) \| \ve r(\tve\theta_k)\|_2.
\end{align*}
Combining these two results yields the upper bound
\begin{equation}\label{eq:lm_inexact_bound}
	\frac{\| (\ma J(\tve \theta_k)^*\ma J(\tve \theta_k) \!+\! \lambda_k^2 \ma I)\tve s_k 
		\!+\! \ma J(\tve \theta_k)^*\ve r(\tve \theta_k)\|_2}
		{\|\ma J(\tve \theta_k)^*\ve r(\tve \theta_k)\|_2}  
		\!\le\! \frac{\sin \phi_{q}(\hset W_k, \hset J_k)\tan \phi_1(\hset J_k, \hset R_k)}
			{\cos^2 \phi_q(\hset W_k, \hset J_k)}
			\frac{\sigma_1(\hma J_k)}{\sigma_q(\hma J_k)}.\!\!\!\!
\end{equation}

This result again confirms the centrality of the subspace angles between $\set W$ and the range of the Jacobian,
although in this result, it is the augmented subspace $\hset W_k$ and augmented Jacobian $\hset J_k$.
By controlling the subspace $\set W_k$, we can ensure that bound in~\cref{eq:lm_inexact_bound}
is smaller than one so that step $\tve s_k$ obeys the bound required by inexact Newton~\cref{eq:inexact_error}.
This suggests that the Levenberg-Marquardt method applied to the projected problem makes progress towards solving the full problem.
However, this result cannot be used online as it requires evaluating the full residual to compute $\phi_1(\hset J_k, \hset R_k)$.
Nor can we use this result to guarantee convergence since
Wright and Holt's convergence result for inexact Levenberg-Marquardt~\cite[Thm.~5]{WH85}
requires an additional sufficient decrease condition.
The projected problem is unlikely to satisfy this additional constraint
since projected problem converges different stationary points
unless the subspace angles between $\set W_k$ and $\set J(\tve\theta_k)$ go to zero as $k\to\infty$.
This prompts the statistical approach we use in the next section
to answer the question: how close are the projected parameter estimates to the full parameter estimates?

%\begin{theorem}[Wright and Holt~{\cite[Thm.~5]{WH85}}]\label{thm:WH}
%	If for any forcing sequence $\lbrace \alpha_k \rbrace_k$, $\alpha_k \le \alpha_0 < 1$
%	there exists a bounded sequence of $\lbrace \lambda_k\rbrace_k$ 
%	such that the sequence $\lbrace \ve x_k\rbrace_k$ satisfying $\ve x_{k+1} = \ve x_k + \ve z_k$
%	where $\ve z_k$ satisfies 
%	\begin{equation}\label{eq:inexact_error}
%		\frac{\| (\ma J(\ve x_k)^*\ma J(\ve x_k) + \lambda_k^2 \ma I)\ve z_k + \ma J(\ve x_k)^*\ve r(\ve x_k)\|_2}
%			{\|\ma J(\ve x_k)^*\ve r(\ve x_k)\|_2} \le \alpha_k < \alpha < 1,
%	\end{equation}
%	and  
%	\begin{equation}\label{eq:WH_decrease}
%		\frac{ \|\ve r(\ve x_k)\|_2^2 - \|\ve r(\ve x_{k+1})\|_2^2}
%			{ \|\ve r(\ve x_k)\|_2^2 - \|\ma J(\ve x_k)\ve z_k + \ve r(\ve x_k) \|_2^2 - \lambda_k^2 \|\ve z_k\|_2^2 }
%			\ge \beta > 0
%	\end{equation}
%	then the sequence $\lbrace \|\ve r(\ve x_k)\|_2\rbrace_k$ is convergent, and the limit points of the sequence $\lbrace \ve x_k\rbrace$
%	are stationary points of $\|\ve r(\cdot )\|_2$.
%\end{theorem}

\section{A statistical perspective\label{sec:statistical}}
One setting in which the nonlinear least squares problem~\cref{eq:nls} can arise
is when measurements $\tve y$ are the sum of $\ve f$ evaluated at some true parameters~$\hve\theta \in \C^q$
 plus Gaussian random noise $\ve g$ with zero mean and covariance~$\epsilon^2\ma I$; 
$\tve y = \ve f(\hve\theta) + \ve g$.
Then the nonlinear least squares problem,
\begin{equation}
	\tve\theta(\ve g) := \argmin_{\ve\theta} \| \ve f(\ve\theta) - (\ve f(\hve\theta) + \ve g) \|_2,
\end{equation}
yields the \emph{maximum likelihood estimate} $\tve\theta$ of $\hve\theta$~\cite[\S2.1]{SW89}.
This estimate has a number of beneficial features.
In the limit of large data or small noise, the estimator $\tve\theta$ is unbiased
and obtains the \emph{Cram\'er-Rao lower bound}, namely, 
$\tve\theta$ has the smallest possible covariance of any unbiased estimator of $\hve\theta$~\cite[\S6.3]{SS10}.
Hence, the corresponding projected parameter estimate
\begin{equation}
	\tve\theta_\set W(\ve g) := \argmin_{\ve\theta} 
		\| \ma P_\set W [\ \ve f(\ve\theta) - (\ve f(\hve\theta) + \ve g)\ ] \|_2
\end{equation}
must have a larger covariance.
By using the inexpensive projected parameter estimate $\tve\theta_\set W$
as an alternative to full parameter estimate $\tve\theta$, 
we are following in a tradition that dates back to Fisher~\cite[\S8]{Fis22}.
Fisher quantified the loss of precision incurred by a particular scalar estimator 
by the \emph{efficiency}: the ratio of the minimum covariance to the estimator's covariance.
For our vector valued estimates, the covariance is a positive definite matrix in $\C^{q\times q}$
and so to obtain a scalar value for the efficiency ratio, 
we follow the lead of experimental design~\cite[\S2.1]{Sil80},\cite[\S1.4]{Mel06}
and consider the determinant of the covariance matrix, which leads to the \emph{D-efficiency}
\begin{equation}\label{eq:true_eff}
	\widehat{\eff}(\set W) := \frac{ \det \Cov \tve\theta}{\det \Cov \tve\theta_\set W} \in [0,1].
\end{equation}
In choosing our subspace $\set W$, our goal will be to make the efficiency as large as possible
so that the covariance of the estimates for $\tve\theta_\set W$ and $\tve\theta$ are similar.
However, as $\tve\theta$ and $\tve\theta_\set W$ are both nonlinear functions of the noise $\ve g$,
we cannot compute a closed form expression for the efficiency.
Instead, following a standard approach for nonlinear experimental design~\cite[\S1.4]{Fed72},
we linearize the parameter estimates $\tve\theta$ and $\tve\theta_\set W$ about $\hve\theta$
yielding the \emph{linearized D-efficiency} as derived in \cref{sec:statistical:linearized}.
The main result of this section will be to connect linearized efficiency to the subspace
angles between $\set W$ and the range of the Jacobian at $\hve\theta$, $\set J(\hve\theta)$:
\begin{equation}\label{eq:linear_eff}
	\eff(\set W, \set J(\hve\theta))
		:= \frac{\det [\ma J(\hve\theta)^*\ma J(\hve\theta)]^{-1}}{
			\det [\ma J(\hve\theta)^* \ma P_\set W \ma J(\hve\theta)]^{-1}}
			= \prod_{k=1}^q \cos^2 \phi_q(\set W, \set J(\hve\theta))
		\approx
			\frac{ \det \Cov \tve\theta}{\det \Cov \tve\theta_\set W}.
\end{equation}
We will refer to $\eff$ as simply the efficiency,
and following Fisher we will say a subspace $\set W$ is $95\%$ efficient for $\hve\theta$
if $\eff(\set W, \set J(\hve\theta)) = 0.95$.
Later in \cref{sec:subspace} we will design subspaces for the exponential fitting problem
with the goal of obtaining a target efficiency. 
Further, note that maximizing efficiency corresponds to minimizing the covariance of $\tve\theta_\set W$
and hence selecting the subspace $\set W$ is similar to experimental design~\cite{Fed72,Sil80,Mel06},
albeit where the design is happening after the data has been collected.

\subsection{Linearized efficiency\label{sec:statistical:linearized}}
Here we briefly derive the linearized estimate of $\tve\theta_\set W$ and the corresponding
linearized covariance; cf.~\cite[\S12.2.6]{SW89}.
In the limit of small noise $\ve g$, we expand $\tve\theta_\set W$ about the true parameters $\hve\theta$
\begin{equation}
	\tve\theta_\set W(\ve g) = \hve\theta + [\ma W^*\ma J(\hve\theta)]^+ \ma W^*\ve g + \order(\|\ve g\|_2^2).
\end{equation}
Applying this first order estimate in the covariance, we have 
\begin{equation}
	\begin{split}
	\Cov\tve\theta_\set W &= 
		\E_{\ve g} [(\hve\theta - \tve\theta_\set W(\ve g))(\hve\theta - \tve\theta_\set W(\ve g))^*] \\
		&\approx \E_\ve g 
			\left[ [\ma W^*\ma J(\hve\theta)]^{+*} \ma W^*\ve g \ve g^*\ma W[\ma W^*\ma J(\hve\theta)]^+\right] 
		= \epsilon^2 [\ma J(\hve\theta)^*\ma W\ma W^*\ma J(\hve\theta)]^{-1}
	\end{split}
\end{equation}
when $\Cov \ve g = \epsilon^2 \ma I$.
To obtain the D-linearized efficiency~\cref{eq:linear_eff},
we replace $\Cov \tve\theta$ and $\Cov\tve\theta_\set W$  with this the estimate above.

%By taking the determinant, the D-efficiency has converted
%The imporant feature about D-efficiency is that it provides an ordering for subspaces.
%Although covariance matrices are positive definite and hence have a partial ordering,
%$\ma A\succeq \ma B$ if $\ma A - \ma B$ is positive definite~\cite[\S7.7]{HJ85},
%this is not a total ordering.
%By taking the determinant, the D-efficiency provides a total ordering of potential subspaces.

%Note that in this linearized case, the Cram\'er-Rao bound follows immediately from linear algebra results.
%the Cram\'er-Rao bound states
%$\Cov\tve\theta_\set W \succeq [\ma J(\hve\theta)^*\ma J(\hve\theta)]^{-1}$~\cite[(6.45),(6.51)]{SS10}.
%Then $[\ma J(\hve\theta)^* \ma W\ma W^*\ma J(\hve\theta)]^{-1} \succeq [\ma J(\hve\theta)^*\ma J(\hve\theta)]^{-1}$
%follows immediately from~\cite[Cor.~7.7.4]{HJ85}.

% PROVIDES TOTAL ORDERING

\subsection{Relating efficiency to subspace angles}
As with the optimization perspective, a good subspace from a statistical perspective
will have small subspace angles between $\set W$ and the Jacobian $\set J(\hve\theta)$.
The following theorem establishes this connection.

\begin{theorem}\label{thm:eff_to_subspace}
	If $\set W$ is an $m$-dimensional subspace of $\C^n$
	with orthonormal basis $\ma W$
	and 
	$\ma J \in \C^{n\times q}$ where $m \ge q$ 
	with $\set J := \Range \ma J$, then
	\begin{equation}
		\eff(\set W, \set J) := 
		\frac{ \det ([\ma J^*\ma J]^{-1})}{\det ([\ma J^*\ma W\ma W^*\ma J]^{-1})} =
		\prod_{k=1}^{q} \cos^2 \phi_k\left(\set W, \set J\right)
	\end{equation}
	where $\phi_k(\set A, \set B)$ is the $k$th principle angle between $\set A, \set B \subset \C^n$
	as defined in~\cref{eq:subspace_angles}. 
\end{theorem}
\begin{proof}
	Let $\ma J = \ma Q \ma T$ be the short-form QR factorization of $\ma J$ where $\ma Q^*\ma Q = \ma I$.
	Then using the multiplicative property of the determinant~\cite[\S0.3.5]{HJ85},
	\begin{align*}
		\eff(\set W, \set J ) &= 
			\frac{ \det(\ma J^*\ma W\ma W^*\ma J)}{\det (\ma J^*\ma J)}
			= \frac{ \det (\ma Q^*\ma W \ma W^*\ma Q) \det (\ma T)\det (\ma T^*)}{
			\det (\ma Q^*\ma Q) \det (\ma T) \det (\ma T^*)}\\
			&= \det (\ma Q^*\ma W\ma W^*\ma Q) 
			= \prod_{k=1}^{q} \sigma_k(\ma W^*\ma Q)^2
			= \prod_{k=1}^q \cos^2 \phi_k(\set W, \set J).
	\end{align*}
\end{proof}

\subsection{Properties of efficiency for projected problems}
We conclude with three results about the linearized D-efficiency
that aid in our construction of subspaces for exponential fitting in \cref{sec:subspace}.
There, our approach will be to precompute a finite number of subspaces for a single exponential
and combine these to produce subspaces for multiple exponentials.
%Although these results are straightfoward,
%they do not appear in the experimental design literature 
%as they rely on our unconventional projection approach.

The first result proves an intuitive fact:
by enlarging the subspace $\set W$,
the efficiency will not decrease.
The following theorem establishes this result,
making use of the partial ordering of positive definite matrices;
namely, $\ma A\succeq \ma B$ if $\ma A - \ma B$ is positive definite~\cite[\S7.7]{HJ85}.

\begin{theorem}\label{thm:worse}
If $\set W_1\subseteq \set W_2$ and $\set J$ are subspaces of $\C^n$ 
then $\eff(\set W_1, \set J)\le \set W_2, \set J)$.
\end{theorem}
\begin{proof}
	Let $\ma W_1$ and $\ma W_2$ be orthonormal bases for $\set W_1$ and $\set W_2$
	and let $\ma Q$ be an orthonormal basis for $\set J$. 
	Then $\ma W_1\ma W_1^*\preceq \ma W_2\ma W_2$ and by \cite[Cor.~7.7.4]{HJ85},
	\begin{equation*}
		\eff(\set W_1, \set J) = \det (\ma Q^*\ma W_1\ma W_1^*\ma Q) \le \det(\ma Q^*\ma W_2\ma W_2^*\ma Q) = 
			\eff(\set W_2, \set J).
	\end{equation*}
\end{proof}

Since our subspaces for multiple exponentials will be built from a union of subspaces for each exponential,
this second result shows that the union satisfies a necessary (but not sufficient)
condition for the combined subspace to have the same efficiency 
as each component subspace had for a single exponential.
\begin{theorem}\label{thm:upper}
	If $\set W$, $\lbrace \set J_k \rbrace_k$ are subspaces of $\C^n$,
		$\set J = \union_{k} \set J_k$, and the dimension of $\set W$ exceeds $\set J$,
		then $\eff(\set W, \set J) \le \min_k \eff(\set W, \set J_k)$.
\end{theorem}
\begin{proof}
	Let $\ma Q = [\ma Q_1,\ma Q_2]$, where $\ma Q_1\in \C^{n\times q_1}$ and $\ma Q_2\in\C^{n\times q_2}$,
	be an orthonormal basis for $\set J$ such that $\ma Q_1\in \C^{n\times q_1}$
	is a basis for $\set J_\ell$ 
	and let $\ma W$ be an orthonormal basis for $\set W$.
	Then as $\sigma_k(\ma W^*\ma Q) \le 1$, 
	\begin{equation*}
		\eff(\set W, \set J) = \prod_{k=1}^{q_1+q_2} \sigma_k(\ma W^*\ma Q)^2 \le 
			\prod_{k=1}^{q_1} \sigma_{q_2 + k}(\ma W^*\ma Q)^2 
			\le \prod_{k=1}^{q_1} \sigma_k(\ma W^*\ma Q_1)^2
			= \eff(\set W, \set J_\ell),
	\end{equation*}
	where the second inequality follows from deleting the last $q_2$ columns of $\ma W^*\ma Q$
	and applying~\cite[Cor.~3.1.3]{HJ91}.
	The result follows by repeating this process for each $\set J_\ell$.
\end{proof}

The final result provides a lower bound on the efficiency for a nearby Jacobian.
This bound is used in \cref{sec:subspace:building} to convert 
a check for efficiency over a continuous set of subspaces $\set J(\ve\theta)$ into 
a check over a discrete set.

\begin{theorem}\label{thm:nearby}
	If $\set W$, $\set J_1$, and $\set J_2$ are subspaces of $\C^n$ and $\set J_1$ and $\set J_2$
	have the same dimension, then
	$\eff(\set W, \set J_2) \eff(\set J_1,\set J_2) \le \eff(\set W,\set J_1)$.
\end{theorem}
\begin{proof}
	Let $\ma Q_1$ and $\ma Q_2$ be orthonormal bases for $\set J_1$ and $\set J_2$.
	As $\ma P_\set W \!\succeq \! \ma P_{\set J_2}\ma P_\set W \ma P_{\set J_2}$
	we obtain the lower bound after application of~\cite[Cor.~7.7.4]{HJ85}
	\begin{multline*}
		\eff(\set W, \set J_1) = \det(\ma Q_1^*\ma P_\set W \ma Q_1) \ge 
			\det(\ma Q_1^*\ma Q_2\ma Q_2^*\ma P_\set W\ma Q_2\ma Q_2^*\ma Q_1)\\
			= \det(\ma Q_2^*\ma P_\set W\ma Q_2)\det(\ma Q_1^*\ma Q_2\ma Q_2^*\ma Q_1)
			= \eff(\set W, \set J_2) \eff(\set J_1, \set J_2).
	\end{multline*}
\end{proof}

With these general results from the two preceding sections complete, 
we now turn to the specifics of the exponential fitting problem.

% Exponential fitting specific results
\section{Fast inner-products for exponential fitting\label{sec:closed}}
In the two preceding sections, we have argued that subspaces $\set W$ should be chosen
such that the subspace angles between $\set W$ and the range of the Jacobian are small.
Now, in this section we turn to the specific problem of selecting a family of subspaces $\set W$ 
for the exponential fitting problem
that not only satisfy this requirement, but also have orthonormal bases $\ma W$
such that projected model $\ma W^*\ve f([\ve\omega, \ve a])$ and 
projected Jacobian $\ma W^*\ma J([\ve\omega,\ve a])$ 
can be inexpensively computed in fewer than $\order(n)$ operations.
For the exponential fitting problem we chose the subspace $\set W(\ve\mu)$
parameterized by $\ve\mu \in \C^m$ with corresponding orthonormal basis $\ma W(\ve\mu) \in \C^{n\times m}$ 
%obtained with a matrix $\ma R(\ve\mu)$ defined in \cref{sec:closed:ortho} 
\begin{equation}
	\set W(\ve\mu) := \Range \ma V(\ve \mu), \qquad \ma W(\ve\mu) := \ma V(\ve\mu)\ma R(\ve\mu)^{-1},
	\qquad \ma R(\ve\mu) \in \C^{m \times m}
\end{equation}
where $\ma V(\ve\mu) \in \C^{n \times m}$ is the Vandermonde matrix $[\ma V(\ve\mu)]_{j,k} = e^{j \mu_k}$
and $\ma R(\ve\mu)$ is constructed as described in \cref{sec:closed:ortho}
such that $\ma W(\ve\mu)$ has orthonormal columns.
We call the parameters $\ve\mu$ \emph{interpolation points} since 
if the entries of $\ve\omega$ are a subset of the entries of $\ve\mu$,
then the projected model interpolates the full model:
\begin{equation}
	\ve \omega \subset \ve\mu \quad \Longrightarrow \quad 
	\ma P_{\set W(\ve\mu)}\ve f([\ve\omega, \ve a]) = 
	\ma P_{\Range \ma V(\ve\mu)} \ma V(\ve\omega) \ve a =
		\ma V(\ve\omega) \ve a = 		 
		\ve f([\ve\omega, \ve a]). 
\end{equation}
In this section we show how to inexpensively compute the product of
$\ma W(\ve\mu)$ with the exponential fitting 
model $\ve f([\ve\omega, \ve a])$ and Jacobian $\ma J([\ve\omega, \ve a])$,
\begin{equation}
	\ma J([\ve\omega, \ve a]) := 
		\begin{bmatrix} \ma V'(\ve\omega) \diag(\ve a) & \ma V(\ve\omega) \end{bmatrix}, \qquad
		[\ma V'(\ve\omega)]_{j,k} = \frac{\partial}{\partial \omega_k} [\ma V(\ve\omega)]_{j,k} = j e^{j \omega_k}.
\end{equation}
Examining the products $\ma W(\ve\mu)^*\ve f([\ve\omega, \ve a])$ and $\ma W(\ve\mu)^*\ma J([\ve\omega, \ve a])$,
\begin{align}
	\ma W(\ve\mu)^*\ve f([\ve\omega, \ve a]) &=
		 \, \ma R(\ve\mu)^{-*}\ma V(\ve\mu)^*\ma V(\ve\omega) \ve a \\
	\ma W(\ve\mu)^*\ma J([\ve\omega, \ve a]) &=
		 \begin{bmatrix}
			\ma R(\ve\mu)^{-*}\ma V(\ve\mu)^*\ma V'(\ve\omega) &
			\ma R(\ve\mu)^{-*}\ma V(\ve\mu)^*\ma V(\ve\omega) \diag \ve a
		\end{bmatrix},
\end{align}
reveals two matrix multiplications of size $n$, 
$\ma V(\ve\mu)^*\ma V(\ve\omega)$ and $\ma V(\ve\mu)^*\ma V'(\ve\omega)$,
that need to be inexpensively computed.
Here we use the geometric sum formula and generalization provided by \cref{thm:polyexp_sum} in \cref{sec:geometric},
to compute the entries of these products in closed form.
Unfortunately, these formulas exhibit catastrophic cancellation in finite precision arithmetic
necessitating careful modifications to obtain high relative accuracy
as described in \cref{sec:closed:geometric} for $\ma V(\ve\mu)^*\ma V(\ve\omega)$
and in \cref{sec:closed:derivative} for $\ma V(\ve\mu)^*\ma V'(\ve\omega)$.
Additionally, we discuss how to compute $\ma R(\ve\mu)$ inexpensively from $\ma V(\ve\mu)^*\ma V(\ve\mu)$
in \cref{sec:closed:ortho}.
The choice of interpolation points $\ve\mu$ is later discussed in \cref{sec:subspace}
and combined with these results yields our algorithm for exponential fitting described in \cref{sec:algorithm}.

\subsection{Geometric sum\label{sec:closed:geometric}}
Each entry in the product of two Vandermonde matrices $\ma V(\ve\mu)^*\ma V(\ve\omega)$
is a geometric sum and hence has a closed form expression via the \emph{geometric sum formula}:
\begin{equation}\label{eq:VV}
	[\ma V(\ve\mu)^*\ma V(\ve\omega)]_{j,k} = 
		\sum_{\ell=0}^{n-1} e^{\conj \mu_j \ell} e^{\omega_k \ell}
		= \begin{cases}
			\displaystyle 
				\frac{ 1 - e^{n(\conj \mu_j + \omega_k)}}{1 - e^{\conj \mu_j + \omega_k}}, & e^{\conj \mu_j + \omega_k}\ne 1; \\
			\displaystyle n, & e^{\conj\mu_j + \omega_k} = 1.
		\end{cases}
\end{equation} 
In finite precision arithmetic this formula exhibits catastrophic cancellation
when $e^{\conj\mu_j + \omega_k} \approx 1$.
Fortunately, many standard libraries provide the special function {\tt expm1} 
that evaluates $e^x - 1$ to high relative accuracy.
However, even with this special function, 
there is still a removable discontinuity at $e^{\conj\mu_j + \omega_k} = 1$.
Hence in floating point we patch this function using a two-term Taylor series expansion around 
$e^{\conj\mu_j + \omega_k} = 1$: 
\begin{equation}
	[\ma V(\ve\mu)^*\ma V(\ve\omega)]_{j,k}  
		= \begin{cases}
			\displaystyle 
				\frac{\expm(n(\conj \mu_j + \omega_k))}{\expm(\conj \mu_j + \omega_k)}, 
					& |\expm(\conj\mu_j + \omega_k)|> 10^{-15}; \\
			\displaystyle n(1 + (n-1)(\conj\mu_j + \omega_k)/2), & |\expm(\conj\mu_j + \omega_k)| \le 10^{-15}.
		\end{cases}
\end{equation}
In our numerical experiments, this expression has a relative accuracy of $\sim \! 10^{-16}$
when compared to a 500-digit reference evaluation of~\cref{eq:VV} using {\tt mpmath}~\cite{mpmath}. 

\subsection{Geometric sum derivative\label{sec:closed:derivative}}
Entries of the product $\ma V'(\ve\mu)^*\ma V(\ve\omega)$ are no longer a geometric sum,
but a \emph{generalized geometric sum}
that has a closed form expression given by \cref{thm:polyexp_sum} in \cref{sec:geometric}:
\begin{equation}\label{eq:VVp_exact}
	\begin{split}
	[\ma V(\ve\mu)^*\ma V'(\ve\omega)]_{j,k}
	&= \sum_{\ell=0}^{n-1} \ell e^{\conj \mu_j \ell} e^{\omega_k \ell} \\
	&= \begin{cases}\displaystyle
		\frac{ -n e^{n (\conj \mu_j + \omega_k )}}{1-e^{ \conj \mu_j + \omega_k }}
			+ \frac{e^{\omega_k +\conj{\mu_j}} (1-e^{n(\omega_k +\conj{\mu_j})})}{(1-e^{\omega_k +\conj{\mu_j}})^2}, 
			& e^{\omega_k +\conj{\mu_j}} \ne 1;\\
			\displaystyle n(n-1)/2, & e^{\omega_k + \conj{\mu_j}} = 1.	
	\end{cases}
	\end{split}
\end{equation}
As with the geometric sum formula, this expression also exhibits catastrophic cancellation
but this can no longer be fixed using standard special functions. 
Instead, we derive a more accurate expression in floating point arithmetic 
by rearranging the expression in the first case and using a Taylor series about $e^{\conj \mu_j + \omega_k} = 1$.
Defining $\delta_{j,k} := \conj \mu_j + \omega_k \in \R\times [-\pi/2,\pi/2)i$
(removing periodicity in the imaginary part)
the first case of \cref{eq:VVp_exact} can be rearranged to yield
\begin{equation}\label{eq:VVp}
	\frac{ -n e^{n \delta_{j,k}}}{1-e^{\delta_{j,k}}}
		+ \frac{e^{\delta_{j,k} } (1-e^{n \delta_{j,k}})}{(1-e^{\delta_{j,k}})^2} 
	= \frac{1 - e^{n\delta_{j,k}}}{1 - e^{\delta_{j,k}}}
	\left[ 
	 \frac{e^{\delta_{j,k}}}{1 - e^{\delta_{j,k}}}
		-\frac{n e^{n \delta_{j,k}}}{1 - e^{n\delta_{j,k}}} 
	\right].
\end{equation}
Although this expression on the right displays even worse catastrophic cancellation
than the expression on the left, 
the first term can be computed using $\tt expm1$
and the expression inside the brackets has a rapidly converging Taylor series:
\begin{multline*}
	%\begin{split}
	\frac{e^{\delta}}{1 - e^{\delta}}
		-\frac{n e^{n \delta}}{1 - e^{n\delta}}
	= \frac{n-1}{2} + \frac{(n^2-1)\delta}{12} - \frac{(n^4-1)\delta^3}{720}
		+ \frac{(n^6-1)\delta^5}{30240} 
	- \frac{(n^8-1)\delta^7}{1209600} \\ + \frac{(n^{10}-1)\delta^9}{47900160}
	- \frac{691(n^{12}-1)\delta^{11}}{1307674368000} + \order(n^{14}\delta^{13}).
	%\end{split}
\end{multline*}
%\end{equation*}
Calling the first seven terms of this expansion the special function $\expdiff(n,\delta)$, 
we then evaluate the product $\ma V(\ve\mu)^*\ma V'(\ve\omega)$ in finite precision arithmetic using 
\begin{equation}
	[\ma V(\ve \mu)^* \ma V'(\ve\omega)]_{j,k}
	= \begin{cases}
		\displaystyle \frac{n e^{n \delta_{j,k}} \expm(\delta_{j,k}) - e^{\delta_{j,k}} \expm(\delta_{j,k} n) }{[\expm(\delta_{j,k})]^2},
			 & |\delta_{j,k}| > 0.5/n;\\
		\displaystyle \frac{\expm(n \delta_{j,k})}{\expm(\delta_{j,k})}\expdiff(n,\delta_{j,k}),	
		  &	0 < |\delta_{j,k}| \le 0.5/n;\\
		\displaystyle n(n-1)/2, & \delta_{j,k} = 0.
	\end{cases}
\end{equation}
In our numerical experiments, 
this expression has a relative accuracy of $\sim\! 10^{-15}$ when compared to a 500-digit reference evaluation
of \cref{eq:VVp} using {\tt mpmath}.

\subsection{Orthogonalization\label{sec:closed:ortho}}
Finally we need to compute the matrix $\ma R(\ve \mu)$ 
such that $\ma V(\ve \mu) \ma R(\ve \mu)^{-1}$ has orthonormal columns inexpensively.
One approach would be to simply take the QR-factorization of $\ma V(\ve\mu )$,
but this is has an $\order(n)$-dependent cost.
Instead, our approach is to form $\ma V(\ve\omega)^*\ma V(\ve\omega)$ using~\cref{eq:VV}
and either take its Cholesky decomposition or its eigendecomposition to compute $\ma R(\ve\mu)$
\begin{align*}
	\ma V(\ve \mu)^* \ma V(\ve\mu) &= \ma R(\ve\mu)\ma R(\ve\mu)^*,  \quad &&\text{(Cholesky)}, \\
	\ma V(\ve \mu)^*\ma V(\ve\mu) &= \ma U(\ve \mu)\ma \Lambda(\ve\mu) \ma U(\ve\mu)^*, 
		\quad &&\text{(eigendecomposition)}, \quad
		\ma R(\ve\mu) = \ma U(\ve\mu) \ma \Lambda(\ve\mu)^{1/2}.
\end{align*} 
Although the Cholesky decomposition should be preferred since $\ma V(\ve\mu)^*\ma V(\ve\mu)$
is positive definite provided each of the $\lbrace e^{\mu_j} \rbrace_j$ are distinct,
in finite precision arithmetic this product can have small negative eigenvalues.
Instead we compute $\ma R(\ve\mu)^{-1}$ using the eigendecomposition, 
truncating the small ($<10^{-14}$) eigenvalues.

\section{A subspace for exponential fitting\label{sec:subspace}}
With the results of the previous section, 
we can now inexpensively project the model and Jacobian onto the subspace $\set W(\ve\mu)$.
However, this leaves one question: how do we choose the interpolation points $\ve\mu$ such that 
the subspace angles between $\set W(\ve\mu)$ and the exponential fitting Jacobian 
\begin{equation}
	\ma J([\ve\omega, \ve a]) = 
		\begin{bmatrix} \ma V'(\ve\omega) \diag(\ve a) & \ma V(\ve\omega) \end{bmatrix}, \qquad
		[\ma V'(\ve\omega)]_{j,k} = \frac{\partial}{\partial \omega_k} [\ma V(\ve\omega)]_{j,k} = j e^{j \omega_k}.
\end{equation}
are small?
Immediately we note that these subspace angles do not depend on $\ve a$ 
(if any entry of $\ve a$ was zero, we would instead fit fewer exponentials);
hence we define
\begin{equation}
	\set J(\ve\omega) := \Range \begin{bmatrix} \ma V'(\ve\omega) & \ma V(\ve\omega) \end{bmatrix}.
\end{equation}
Structurally, it may seem impossible to have small subspace angles between $\set W(\ve\mu)$ 
and $\set J(\ve\omega)$ since $\set W(\ve\mu) = \Range \ma V(\ve\mu)$ does not contain any columns from $\ma V'$.
However, since columns of $\ma V'$ are the derivatives of the columns of $\ma V$,
we can approximate the range of$ \ma V'$ using finite differences
\begin{equation}
	\ma V(\omega) \approx \frac{ \ma V(\omega + \delta_+) + \ma V(\omega + \delta_-)}{2} \qquad
	\ma V'(\omega) \approx \frac{ \ma V(\omega + \delta_+) - \ma V(\omega + \delta_-)}{\delta_+ - \delta_-}.
\end{equation}
Hence, for an appropriate choice of interpolation points, the subspace angles between $\set W(\ve\mu)$
and $\set J(\ve\omega)$ are small;
for example, the interpolation points 
$\delta_\pm(\omega) = 0.8 \Re\omega \pm \max \lbrace -0.52\Re\omega, 1.39/n \rbrace$
yields a subspace with $95\%$ efficiency for any $\omega\in (-\infty,0]\times [-\pi,\pi)i$.
Here our approach for selecting interpolation points is to divide the parameter space 
for a single exponential with frequency $\omega \in (-\infty,0]\times [-\pi,\pi)$ into a series of boxes
as shown in \cref{fig:box_ring}.
These boxes have been constructed such that when the corners of the box containing $\omega$
are taking as interpolation points, the efficiency of this subspace is at least $95\%$
and hence the subspace angles between $\set W(\ve\mu)$ and $\set J(\ve\omega)$ are small.
Then, for multiple exponentials, we simply combine the subspaces generated by this heuristic,
justified by~\cref{thm:upper} that this combination is a necessary condition
for the combined subspace to also have at least $95\%$ efficiency.
There are several advantages to this box partition approach.
By limiting ourselves to a finite number of interpolation points,
we can frequently reuse the multiplication $\ma V(\mu_j)^*\tve y$ as the subspace updates.
Moreover, constructing this subspace is inexpensive, 
allowing frequent updates during optimization.
In the remainder of this section we first discuss a practical algorithm for building this box partition
for any target efficiency and then give the coordinates for box partition with $95\%$ efficiency. 

\begin{figure}
\begin{center}
\begin{tikzpicture}
	\begin{groupplot}[
			group style = {
				group size =2 by 1,
				horizontal sep = 6em,
			},
			width = 0.4\linewidth,
		]
		\nextgroupplot[xmin=-2, xmax=0, ymin=-0.5, ymax=0.5,
			xlabel=$\Re\omega$, ylabel=$\Im\omega$,
			ytick = {-0.5, -0.25, 0, 0.25,0.5},
			yticklabels = {$-\pi$, $-\pi/2$, $0$, $\pi/2$, $\pi$},
			xtick = {-1.447, -0.7181, -0.3812,-0.1924,0},
			xticklabels = {$\alpha_1$, $\alpha_2$, $\alpha_3$, $\alpha_4$, $0$},
			%x tick label style = {rotate = 90} ,
			width = 0.45\linewidth,
			height = 0.45\linewidth,
		]

		\addplot[colorbrewerA2, fill, fill opacity = 0.5, draw = none ] table [x=x,y expr = \thisrow{y}/(2*pi)] {fig_box_ring_eff.dat}; 
		% Draw the selected box
		\addplot[black, mark=*, only marks] coordinates 
			{(-0.7181,0) (-0.3812,0) (-0.3812,1/8) (-0.7181, 1/8) };
		\addplot[fill opacity=1,fill =colorbrewerA1] coordinates 
			{(-0.7181,0) (-0.3812,0) (-0.3812,1/8) (-0.7181, 1/8) };

		\addplot[black, no marks] coordinates {(0,-0.5) (0,0.5)};%
		\foreach \n in {0, ..., 32}{%
			\addplot[black, no marks] coordinates {(0,\n/32 - 0.5) (-0.1924,\n/32 -0.5)};%
		} %
		\addplot[black, no marks] coordinates {(-0.1924,-0.5) (-0.1924, 0.5)};%
		\foreach \n in {0, ..., 16}{%
			\addplot[black, no marks] coordinates {(-0.1924,\n/16 - 0.5) (-0.3812,\n/16 -0.5)};%
		} %
		\addplot[black, no marks] coordinates {(0,-0.5) (0,0.5)};%
		\foreach \n in {0, ..., 16}{%
			\addplot[black, no marks] coordinates {(0,\n/16 - 0.5) (-0.3812,\n/16 -0.5)};%
		} %
		\addplot[black, no marks] coordinates {(-0.3812,-0.5) (-0.3812, 0.5)};%
		\foreach \n in {0, ..., 8}{%
			\addplot[black, no marks] coordinates {(-0.3812,\n/8-0.5) (-0.7181,\n/8-0.5)};%
		} %
		\addplot[black, no marks] coordinates {(-0.7181,-0.5) (-0.7181,0.5)};%
		\foreach \n in {0, ..., 4}{%
			\addplot[black, no marks] coordinates {(-1.447,\n/4-0.5) (-0.7181,\n/4-0.5)};%
		} %
		\addplot[black, no marks] coordinates {(-1.447,-0.5) (-1.447,0.5)};%
		\foreach \n in {0, ..., 2}{%
			\addplot[black, no marks] coordinates {(-1.447,\n/2-0.5) (-5,\n/2-0.5)};%
		}% 
		\node (box-right) at (axis cs: 0.2, 0) {};

		\addplot[black, mark=x, thick, mark size=4pt] coordinates {(-0.5,0.07)};

		\nextgroupplot[
		xmin=-1.01,ymin=-1.01, xmax=1.1, ymax=1.01,
		width=0.46\linewidth, height=0.45\linewidth,
		 axis lines=none,
		axis equal,
		xticklabels = {},
		yticklabels = {},
		]

		% Plot the contours
		\addplot[colorbrewerA2, fill, fill opacity = 0.5, draw = none ] table [x=x,y = y] {fig_box_ring_exp_eff.dat}; 
		% Draw the selected box
		\draw[fill=colorbrewerA1, fill opacity = 1] 
			(axis cs: 0.4928,0) -- (axis cs:0.6916,0) arc 
				[start angle=0, end angle=45, radius={transformdirectionx(0.6916)}]
			-- (axis cs:{0.4928*cos(45)}, {0.4928*sin(45)})
				arc [start angle = 45, end angle=0, radius={transformdirectionx(0.4928)}]
			-- cycle;
		\addplot[black, mark=*, only marks] coordinates 
			{({exp(-0.7181)},0) ({exp(-0.3812)},0) ( {exp(-0.3812)*cos(45)}, {exp(-0.3812)*sin(45)} ) ( {exp(-0.7181)*cos(45)}, {exp(-0.7181)*sin(45)}) };
		\draw (axis cs:0,0) circle [black, radius=0.2221];
		\draw (axis cs:0,0) circle [black, radius=0.4928];
		\draw (axis cs:0,0) circle [black, radius=0.6916];
		\draw (axis cs:0,0) circle [black, radius=0.8246];
		\draw (axis cs:0,0) circle [black, radius=1.0];
		\node (ring-left) at (axis cs:-1.2, 0) {};

		% Inner ring
		\addplot[black, no marks] coordinates {(-0.2221, 0) (0.2221,0)};
		\foreach \n in {0,...,3}{
			\addplot[black, no marks] coordinates {
				({0.2221*cos(360*\n/4)}, {0.2221*sin(360*\n/4)}) 
				({0.4928*cos(360*\n/4)}, {0.4928*sin(360*\n/4)})
				};
		}
		\foreach \n in {0,...,7}{
			\addplot[black, no marks] coordinates {
				({0.4928*cos(360*\n/8)}, {0.4928*sin(360*\n/8)})
				({0.6916*cos(360*\n/8)}, {0.6916*sin(360*\n/8)}) 
				};
		}
		\foreach \n in {0,...,15}{
			\addplot[black, no marks] coordinates {
				({0.6916*cos(360*\n/16)}, {0.6916*sin(360*\n/16)}) 
				({0.8246*cos(360*\n/16)}, {0.8246*sin(360*\n/16)}) 
				};
		}
		\foreach \n in {0,...,31}{
			\addplot[black, no marks] coordinates {
				({0.8246*cos(360*\n/32)}, {0.8246*sin(360*\n/32)}) 
				({1*cos(360*\n/32)}, {1*sin(360*\n/32)}) 
				};
		}
		\addplot[black, mark=x, thick, mark size=4pt] coordinates {(0.54880534905805234,0.25824819460494125)};

	\end{groupplot}
		\draw[black, thick, ->] ([xshift=5pt]$(group c1r1.east)$) -- ([xshift=-5pt]$(group c2r1.west)$) node[midway, above]
			{$x\mapsto e^x$};
\end{tikzpicture}
\end{center}
\caption{An illustration of the box partition of the parameter space 
	$\ve\omega \in (-\infty,0] \times [-\pi,\pi)i$ for exponential fitting. 
	For a given exponential with frequency $\omega$ denoted by $\times$,
	we select the box containing it, denoted in red, and 
	the corresponding four interpolation points at the corners, denoted by $\bullet$.
	The blue shaded region shows the set of $\omega$ 
	where the efficiency at $\omega$ using this subspace is at least $95\%$.
	This region includes the entire box and extends outward.
	The figure on the right shows the same features under the exponential map,
	exposing the periodicity of the parameter space.
}
\label{fig:box_ring}
\end{figure}
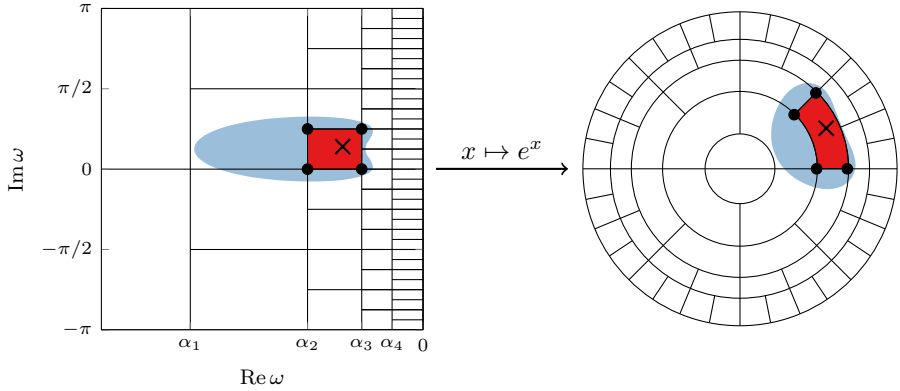

\subsection{Building the partition\label{sec:subspace:building}}
Our goal in constructing the box partition is to guarantee that the target efficiency
$\eff_\text{target}$ is obtained for every single exponential with frequency $\omega$ inside the box.
This is an expensive task, so we build our boxes to simplify the verification process.
As shown in \cref{fig:box_ring}, our box partition consists of a series of stacks
where each stack has twice as many boxes as the one on its left, 
evenly dividing the imaginary component of the parameter space.
Then, as efficiency depends on $\delta_{j,k} = \conj \mu_j + \omega_k$,
cf.~\cref{eq:VV} and \cref{eq:VVp_exact},
simultaneously shifting the imaginary parts of $\mu_j$ and $\omega_k$ does not change $\delta_{j,k}$
and hence verifying that one box in the stack obtains the target efficiency for each $\omega$ inside
establishes the same for the remaining boxes in the stack.
With this construction, there is only one set of free parameters:
the real coordinates of each box $\lbrace \alpha_\ell\rbrace_{\ell\ge 0}$.
We choose these $\alpha_\ell$, starting from $\alpha_0 = -\infty$,
by making $\alpha_\ell$ as large as possible while still obtaining the target efficiency inside each box:
\begin{equation}\label{eq:alpha_design}
	\begin{split}
		\alpha_\ell = \maximize_{\alpha > \alpha_{\ell-1}} \ \alpha \quad
		\text{such that} &\quad 
			\eff_{\text{min}} \le \minimize_{\omega\in [\alpha_{\ell-1}, \alpha]\times [0,2\pi/2^\ell]i} 
			\eff(\set W(\ve\mu), \set J(\omega)) \\
		\text{where} & \quad 
			\ve\mu = \begin{bmatrix} \alpha_{\ell-1} & \alpha_{\ell-1}+2\pi/2^{\ell}i &
		 \alpha & \alpha+2\pi/2^{\ell}i \end{bmatrix}.
	\end{split}
\end{equation}
This is a challenging nested optimization problem over a continuous set of $\omega$,
so we invoke \cref{thm:nearby} to construct an auxiliary grid of exponentials
where obtaining a slightly higher efficiency at these discrete points guarantees the target efficiency
is reached for any $\omega$ inside the box.

To construct this auxiliary grid, we specify a series of real parts $a_j \in \R$ and 
imaginary spacings $b_j\in \R_+$ that define the grid points $z_{j,k} := a_j + ik b_j$.
To specify $a_j$ and $b_j$ with a grid efficiency of $\eff_\text{grid}$,
starting from $a_0 = \alpha_\ell$ we solve the single variable finding root problem
that yields $a_j$ and $b_j$,
\begin{equation}
	\eff(\set J(a_j), \set J(a + (1+1i)c)) = \eff_{\text{grid}}, 
		\quad \Rightarrow \quad  a_{j+1} := a_j + c, \quad b_{j+1} := c,
\end{equation}
setting $b_0 = b_1$.
Then, invoking~\cref{thm:nearby}, we have the bound
\begin{equation}\label{eq:alpha_design_grid}
			\minimize_{\omega\in [\alpha_{\ell-1}, \alpha]\times [0,2\pi/2^\ell]i} 
			\eff(\set W(\ve\mu), \set J(\omega)) \le 
			\minimize_{\substack{\!\!\!\!\!\!\!\! 
				j,k \in \Z_+ \\ z_{j,k} \in [\alpha_{\ell-1}, \alpha]\times [0,2\pi/2^\ell]i} 
				\!\!\!\!\!\!}
			\eff_{\text{grid}} \cdot \eff(\set W(\ve\mu), \set J(z_{j,k})).\!\!\!
\end{equation}
Substituting this bound in~\cref{eq:alpha_design} replaces the inner optimization
with finding the minimum over a discrete set, simplifying the problem.
Further, since the accuracy of the efficiency computation is limited by the grid,
we restrict the maximization over $\alpha$ to the discrete set of grid points $a_j$.

\subsection{The ninety-five percent efficiency partition}
Here we provide the coordinates for a box partition with a target efficiency of $95\%$
constructed using $\eff_{\text{grid}} = 0.99999$ in \cref{tab:rings}
for multiple values of $n$.
In practice, we restrict our interpolation points to the closed left half plane
and hence set the first $\alpha_\ell$ greater than zero to zero.
Although this choice of target efficiency was arbitrary, 
it does make the right-most interpolation points 
correspond to the $n$th roots of unity that appear in the discrete Fourier transform (DFT).

\begin{table}
\caption{The real coordinates of the box partition $\alpha_\ell$
	as determined by solving \cref{eq:alpha_design} with $\eff_\text{grid} = 0.99999$.
	Due to the limitation of this discretization,
	these values are only accurate to approximately three digits.}
\label{tab:rings}

\begin{center}
	\footnotesize
	\begin{filecontents*}{tab_rings.dat}
ell	inf	n16	n1024	n256	n1048576	
1.0	-1.42050651219	-1.42050651219	-1.42050651219	-1.42050651218	-1.42050651219	
2.0	-0.666693304632	-0.666692744913	-0.666693304631	-0.66669330463	-0.666693304631	
3.0	-0.352861035783	-0.347971461104	-0.352861035784	-0.352861035784	-0.352861035783	
4.0	-0.181881377592	0.0812069407346	-0.181881377592	-0.181881377593	-0.181881377592	
5.0	-0.0919772497042	nan	-0.091977249705	-0.091977249716	-0.0919772497042	
6.0	-0.0461687556534	nan	-0.0461687556539	-0.0461687516769	-0.0461687556534	
7.0	-0.0231321458827	nan	-0.0231321458968	-0.0229430447027	-0.0231321458826	
8.0	-0.011565270769	nan	-0.011565269813	0.00355169216022	-0.0115652707689	
9.0	-0.00578237936891	nan	-0.00574811444843	nan	-0.00578237936889	
10.0	-0.00289107980963	nan	0.000871336108073	nan	-0.00289107980963	
11.0	-0.00144548723513	nan	nan	nan	-0.00144548723512	
12.0	-0.000722717566851	nan	nan	nan	-0.00072271756685	
13.0	-0.000361345790255	nan	nan	nan	-0.00036134575373	
14.0	-0.000180666373096	nan	nan	nan	-0.00018066644118	
15.0	-9.03300100268e-05	nan	nan	nan	-9.03299911766e-05	
16.0	-4.51634168781e-05	nan	nan	nan	-4.51633754722e-05	
17.0	-2.25809144085e-05	nan	nan	nan	-2.25808587873e-05	
18.0	-1.12900602044e-05	nan	nan	nan	-1.12900206576e-05	
19.0	-5.6448316099e-06	nan	nan	nan	-5.61133420123e-06	
20.0	-2.82231656255e-06	nan	nan	nan	8.605e-07	
\end{filecontents*}
\pgfplotstabletypeset[
		clear infinite, 
		precision = 3,
		sci,
		sci zerofill,
		dec sep align,
		columns/ell/.style={int detect, column name = {$\ell$}},
		columns/inf/.style={column name = {$n=\infty$}},
		columns/n16/.style={column name = {$n=16$}},
		columns/n256/.style={column name = {$n=256$}},
		columns/n1024/.style={column name = {$n=1024$}},
		columns/n1048576/.style={column name = {$n=2^{20}$}},
		columns = {ell, n16, n256, n1024, n1048576, inf},
		every head row/.style = {
			before row = {\toprule},
			after row = {\midrule},
		},
		every last row/.style={after row=\bottomrule},
	]{tab_rings.dat}
\end{center}

\end{table}

Although \cref{tab:rings} only displays the coordinates for several values of $n$,
two patterns emerge that allow us to estimate the box partition for any $n$.
First note that the values for $\alpha_\ell$ when $n\ne \infty$ match those for $n=\infty$
for all but the last two, which are always larger.
Hence the values of $\alpha_\ell$ for $n=\infty$ are an lower bound on those for arbitrary $n$.
The other pattern is that after the first five, the $\alpha_\ell$ for $n=\infty$ 
shrink exponentially with  
\begin{equation}
	\alpha_\ell \approx -2.9720 \cdot 2^{-\ell}; \quad \ell \ge 5.
\end{equation}
These two patterns allow us to pick the box partition using $\alpha_\ell$ for $n=\infty$,
extending this sequence using the approximation above for larger values of $\ell$.

When $n$ is not a power of two, the box partition 
will no longer have the $n$th roots of unity available as interpolation points.
Due to their connection with the discrete Fourier transform, we prefer to keep $n$th roots of unity available
and thus modify the construction of the box partition.
For an $n$ that is not a power of two,
everything remains the same except when $\omega$ is in the rightmost stack,
$\Re \omega \in (\alpha_{\widehat\ell},0]$.
In this case we no longer use boxes, but pick the two closest interpolation points
with $\Re \mu = \alpha_{\widehat \ell}$ from the leftward stack 
and the two closest $n$th roots of unity where $\Re\mu = 0$.
Although we are no longer able to guarantee $95\%$ efficiency for exponentials in this range,
this heuristic still provides a high efficiency subspace in practice.

\section{A projected exponential fitting algorithm\label{sec:algorithm}}
Equipped with subspace $\set W(\ve\mu)$ for which the projected model $\ma W(\ve\mu)^*\ve f([\ve\omega, \ve a])$
and Jacobian $\ma W(\ve\mu)^*\ma J([\ve\omega, \ve a])$ 
can be inexpensively computed as described in \cref{sec:closed}
and combined with the heuristic from \cref{sec:subspace} to pick interpolation points $\ve\mu$
so that the subspace angles between $\set W(\ve\mu)$ and the range of the Jacobian $\set J(\ve\omega)$ are small,
we now construct an algorithm to solve exponential fitting problem using projected nonlinear least squares. 
Our basic approach is to solve a sequence of projected nonlinear least squares problems using Levenberg-Marquardt
and infrequently update the subspace during the course of optimization.
In this section we describe several important details for this algorithm.
First, in \cref{sec:algorithm:varpro} we show how \emph{variable projection}~\cite{GP03,GP73}
can be used to implicitly solve for the amplitudes $\ve a$ 
revealing optimization problem over frequencies $\ve\omega$ alone;
then in \cref{sec:algorithm:subspace} we discuss the details of how to update the subspace;
and finally in \cref{sec:algorithm:initial} we show how to obtain initial estimates of the frequencies $\ve\omega$
to enable a fair comparison with subspace based methods which do not require initial estimates.

\subsection{Variable projection\label{sec:algorithm:varpro}}
The key insight behind variable projection originated in a PhD thesis by Scolnik 
on the exponential fitting problem~\cite{Sco70}.
Recognizing the optimal linear parameters $\ve a$ are given by the pseudoinverse
for a fixed $\ve\omega$, $\ve a = \ma V(\ve\omega)^+\tve y$,
allows the residual to be stated as a function of $\ve\omega$ alone:
\begin{equation}
	\ve r([\ve\omega, \ve a]) = \ve f([\ve\omega, \ve a]) - \tve y = \ma V(\ve\omega) \ve a - \tve y \Rightarrow
	\left[\ma V(\ve\omega)\ma V(\ve\omega)^+ - \ma I \right]\tve y = \ma P_{\ma V(\ve\omega)}^\perp \tve y,
\end{equation}
where $\ma P_{\ma V(\ve\omega)}^\perp$ is the orthogonal projector onto the subspace perpendicular
to the range of $\ma V(\ve\omega)$.
This allows us to define an equivalent optimization problem over $\ve\omega$ alone
\begin{equation}
	\minimize_{\ve\omega \in \C^p} \| \hve r(\ve\omega) \|_2^2, \qquad 
		\hve r(\ve\omega) := \ma P_{\ma V(\ve\omega)}^\perp \tve y.
\end{equation}
Golub and Pereyra~\cite{GP73} provide the Jacobian for this variable projection residual $\hve r$, 
\begin{equation}
	[\hma J(\ve\omega)]_{\cdot,k} := -\left[
			\ma P_{\ma V(\ve\omega)}^\perp \frac{\partial \ma V(\ve\omega)}{\partial \omega_k} \ma V(\ve\omega)^-
			+\ma V(\ve\omega)^{-*}\left(\frac{\partial \ma V(\ve\omega)}{\partial \omega_k}\right)^*\ma P_{\ma V(\ve\omega)}^\perp
		\right] \tve y,
\end{equation}
and we can further exploit the structure of the exponential fitting problem 
to reveal a simple expression for this Jacobian.
Defining the short form QR-factorization of $\ma V(\ve\omega)$, 
$\ma V(\ve\omega) = \ma Q(\ve\omega) \ma T(\ve \omega)$,
this Jacobian becomes  
\begin{equation}
	\hma J(\ve\omega) = \left[\ma I - \ma Q(\ve\omega)\ma Q(\ve\omega)^*\right]\ma V'(\ve\omega)\diag(\ve a)
		 - \ma Q(\ve\omega)\ma T(\ve\omega)^{-*}\diag(\ma V'(\ve\omega)^*\hve r(\ve\omega)).
\end{equation}
With the linear parameters removed, the Levenberg-Marquardt method can then be applied to
the variable projection residual $\hve r(\ve\omega)$ and Jacobian $\hma J(\ve\omega)$.
The same expressions also apply to the projected problem
upon making the substitutions: $\ma V(\ve\omega)\to \ma W(\ve\mu)^*\ma V(\ve\omega)$,
$\ma V'(\ve\omega)\to \ma W(\ve\mu)^*\ma V'(\ve\omega)$, and $\tve y\to \ma W(\ve\mu)^*\tve y$.

\subsection{Updating subspaces\label{sec:algorithm:subspace}}
As the analysis in~\cref{sec:optimization} suggests that the subspace angles between
$\set W(\ve\mu)$ and $\set J(\ve\omega)$ need to remain small, 
we repeatedly update the interpolation points during the course of optimization.
Here we use the efficiency based heuristic described in \cref{sec:subspace}
to pick interpolation points, as a high efficiency ensures small subspace angles between 
$\set W(\ve\mu)$ and $\set J(\ve\omega)$.
However, rather than discarding interpolation points no longer required by this heuristic,
we preserve them, continually expanding the subspace. 
This is necessary to prevent the optimization algorithm from entering a cycle.

\subsection{Initialization\label{sec:algorithm:initial}}
A final issue concerns how we provide the initial values the optimization algorithm.
Subspace based methods do not require these initial values
and so to provide a fair comparison we use a simple initialization heuristic.
It is well known that peaks in the discrete Fourier transform (DFT) of a signal, 
$\ma F_n^*\tve y$ where $[\ma F_n]_{j,k} = n^{-1/2}e^{2\pi i j k/n}$,
correspond to the frequencies present~\cite{SM97}.
This forms the foundation of many initialization approaches.
For example, in magnetic resonance spectroscopy these peaks can be identified manually~\cite[\S3.3]{VBH96}
to initialize an optimization algorithm.
Here, we pick the initial estimate iteratively.
Starting with the first exponential, we set $\omega_1 = 2\pi i \widehat{k}/n$
where $\widehat k$ is the largest entry in $\ma F_n^*\tve y$.
Then after the optimization algorithm has terminated, we initialize $\omega_2$
based on the largest entry of the residual $\ma F_n^*\hve r(\ve\omega)$.
This process repeats until the desired number of exponentials have been recovered.
This approach is similar to that of Macleod~\cite{Mac98}, 
but we optimize all the frequencies $\ve\omega$ at each step.

\section{A numerical example\label{sec:example}}
To demonstrate the effectiveness of projected nonlinear least squares for exponential fitting,
we apply the algorithm described in \cref{sec:algorithm}
to a magnetic resonance spectroscopy test problem from~\cite[Table~1]{VBH97};
see, e.g., \cite[\S12.4]{HPS13} for a discussion of the underlying physics.
This example describes a continuous complex signal $y(t)$ consisting of eleven exponentials:
\begin{equation}\label{eq:nmr}
	\begin{split}
	&y(t) = \sum_{k=1}^{11} a_k e^{135i \pi / 180} e^{(2i\pi f_k - d_k) t} \quad \text{ where } \\
	&\begin{matrix} 
	\ve a &= [& 75 & 150 & 75 & 150 & 150 & 150 & 150 & 150 & 1400 & 60 & 500 &]\phantom{.} \\
	\ve f &= [& -86&-70&-54&152&168&292&308&360&440&490&530 &]\phantom{.}\\
	\ve d &= [& 50&50&50&50&50&50&50&25&285.7&25&200 &]
	\end{matrix}
	\end{split}
\end{equation}
from which we construct measurements $\tve y \in \C^n$ by 
sampling $y(t)$ uniformly in time and contaminating these
with independent and identically distributed additive Gaussian noise $\ve g$ with $\Cov \ve g = \ma I$
according to the formula 
\begin{equation}
	[\tve y]_k = y(\delta(n) k) + 15 [\ve g]_k, \qquad \delta(n): = \frac{256}{3n}\cdot 10^{-3}.
\end{equation}
This allows us to scale the original problem which took $n=256$ 
by increasing the sample rate $\delta$.
In this section we consider three algorithms applied to this exponential fitting problem:
conventional nonlinear least squares, our projected nonlinear least squares,
and HSVD~\cite{BBO87} as a representative of subspace based methods due to its simple implementation.
We present our results using two different implementations of HSVD:
an implementation using dense linear algebra and  fast implementation
using an $\order(n\log n)$ Hankel matrix-vector product and an iterative SVD algorithm following~\cite{LMVHH02}.
Our goal is to compare these algorithms on two metrics:
the wall clock time taken to solve the exponential fitting problem
and the precision of the resulting parameter estimates.
A Matlab implementation of our projected nonlinear least squares algorithm for exponential fitting,
the two HSVD implementations described,
and code to construct these examples are provided at {\tt \url{https://github.com/jeffrey-hokanson/ExpFit}}.
In these implementations we use tight convergence tolerances:
$10^{-16}$ for both residual norm and solution change in Matlab's
nonlinear least squares solver {\tt lsqnonlin} 
and $10^{-16}$ for the Ritz residual in Matlab's {\tt eigs} used in the fast HSVD implementation.

\begin{figure}
	\centering
	\begin{tikzpicture}[baseline=(current bounding box.north)]
		\pgfplotstableread{fig_timing_nmr_full.dat}\full
		\pgfplotstableread{fig_timing_nmr_hsvd_fast.dat}\hsvdfast
		\pgfplotstableread{fig_timing_nmr_projected.dat}\projected
		\pgfplotstableread{fig_timing_nmr_hsvd.dat}\hsvd
		\pgfplotstableread{fig_timing_nmr_mat-vec.dat}\matvec
		\begin{loglogaxis}[
			xlabel = {data dimension $n$},
			ylabel = wall time (seconds),
			width=0.94\textwidth,
			height=0.47\textwidth,
			title = Wall clock performance comparison,
			xmin = 1e2,
			xmax = 20e6,
			ymin = 1e-2,
			ymax = 1e3,
			ytickten = {-3,-2,-1,0,1,2,3}
			]
			% Median lines
			\addplot[very thick, colorbrewerA1 ] table [x=n, y = median]{\projected}
				[anchor = south east] node [pos=0.95, rotate=23] {\bf projected NLS};
			\addplot[very thick, colorbrewerA2, dashdotted] table [x=n, y = median]{\full}
				[anchor = south east] node [pos =0.95, rotate=28] {full NLS};
			\addplot[very thick, colorbrewerA3, dotted] table [x=n, y = median]{\hsvd}
				[anchor = south east] node [pos =1.0, rotate=53] {dense HSVD};
			\addplot[very thick, colorbrewerA4, densely dotted] table [x=n, y = median]{\hsvdfast}
				[anchor = south east] node [pos =0.95, rotate=29] {fast HSVD};%
			\addplot[very thick, colorbrewerA5, dashdotdotted] table [x=n, y expr = \thisrow{mean}]{\matvec}
				[anchor = north west] node [pos =0.85, rotate=28, yshift=-2pt] {$\ma V(\ve\mu)^*\tve y$};%

			%\addplot[thick, black, <->] coordinates { (8.3886e6,2.43e1) (8.3886e6, 3.72e2)}
			%	[anchor = west] node [pos= 0.5] {$16\times$};
			\addplot[thick, black, <->] coordinates { (1.6777e7,5.453e1) (1.6777e7,5.78e2)}
				[anchor = east] node [pos= 0.4] {$10\times$};
			
			\addplot[thick, black, <->] coordinates { (5.24e5,2.025 ) (5.24e5, 2.674e2)}
				[anchor = east] node [pos= 0.60] {$140\times$};
	
			% Timing triangles	
			%\addplot[black] coordinates {(2e3,10) (2e3,640) (8e3,640) (2e3,10)}
			%	[anchor = east ]node[pos=.3] {3};
			\addplot[thick, black, ->] coordinates { (2e3, 10) (8e3,640)}
				[anchor = south west ] node [pos = 0.3, rotate = 58] {$\order(n^3)$};
			\addplot[thick, black, ->] coordinates { (1e6, 0.2) (1e7,2)}
				[anchor = north west ] node [pos = 0.35, rotate = 30, yshift =-2pt] {$\order(n)$};
		\end{loglogaxis}
	\end{tikzpicture}
	\caption{
		The median wall clock time from from ten runs 
		for four different exponential fitting algorithm implemented in Matlab 2016b,
		applied to data from~\cref{eq:nmr},
		and running on a 2013 Mac Pro with a 3.5 GHz 6-Core Intel Xeon E5 and 16 GB of RAM clocked at 1866 MHz.
		We also show the time taken to form $\ma V(\ve\mu)^*\tve y$ for $\ve\mu \in \C^{44}$
		which is a lower bound on the time taken by our projected nonlinear least squares algorithm;
		this is approximately the time required to check the first order necessary conditions.
	}
	\label{fig:timing_nmr}
\end{figure}
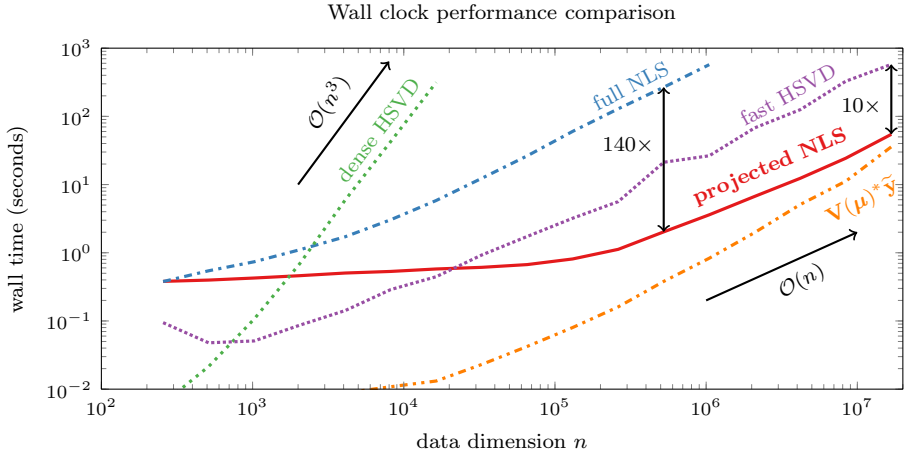

\subsection{Timing}
As these three algorithms use different paradigms for solving the exponential fitting problem,
we compare their performance using total wall clock time.
\Cref{fig:timing_nmr} shows the time taken by each algorithm when 
applied to the magnetic resonance spectroscopy test problem given in \cref{eq:nmr}.
Asymptotically, the time taken by the dense HSVD implementation scales like $\order(n^3)$ due to the dense SVD,
whereas the fast HSVD implementation scales like $\order(n\log n)$ due to the use of the fast Fourier transform (FFT)
to compute the Hankel matrix-vector product.
Similarly, the time taken both nonlinear least squares based approaches scales like $\order(n\log n)$
due to their use of the FFT in the initialization heuristic.
Although these three algorithms each have the same asymptotic rate, 
their constants are different.
In the limit of large data, the projected nonlinear least squares implementation is fastest,
but for small data, the repeated initialization of the optimization algorithm dominates the cost.
It is possible that a more careful implementation could avoid this cost
and bring the wall clock time for projected nonlinear least squares to closer,
or perhaps faster than, the fast HSVD implementation in the limit of small data.

\subsection{Precision\label{sec:example:precision}}
In addition to providing faster performance than fast HSVD for large data,
the projected nonlinear least squares approach also yields more precise parameter estimates.
Considering the same magnetic resonance spectroscopy example,
we seek to quantify the precision of our parameter estimates.
In the limit of small noise, the error in the parameter estimate $\tve\theta = [\tve\omega, \tve a]$
relative to the true parameters $\hve\theta = [\hve\omega, \hve a]$ is normally distributed with covariance:
\begin{equation}\label{eq:asymptotic_cov}
	\ma \Gamma := \ma J(\hve\theta)^*\E_{\ve g}[\ve g\ve g^*]\ma J(\hve\theta)
				= \ma J(\hve\theta)^*\ma J(\hve\theta) \cdot \E_\ve g[\|\ve g\|_2^2] \approx \Cov \tve \theta.
\end{equation}
If $\ma \Gamma$ is actually the covariance of $\tve\theta$, then $\ma \Gamma^{-1/2}(\tve\theta - \hve \theta)$
is normally distributed with zero mean an unit variance.
Hence, the norm of the mismatch $\|\ma \Gamma^{-1/2}(\tve \theta - \hve\theta)\|_2^2$
follows a $\chi^2$ distribution with $44$ degrees of freedom.
As seen in \cref{fig:nmr_pert},
the distribution of error of the projected nonlinear least squares problem approximately matches that of the full problem
and approaches the desired $\chi^2$ distribution as $n$ becomes large.
However, HSVD provides less precise parameter estimates, a result that follows from the analysis of Rao~\cite{Rao88}.

\begin{figure}
	\centering
	\begin{tikzpicture}
		\pgfplotstableread{fig_nmr_pert_chi2.dat}\kde
		\begin{groupplot}[
			group style = {
				group size=3 by 2,
				ylabels at= edge left,
				xlabels at = edge bottom,
				yticklabels at =edge left,
				horizontal sep=15pt,
				vertical sep = 15pt,
			},
			ylabel style={text width=3.5cm, text centered},
			width=0.37\textwidth,
			height=0.29\textwidth,
			xmin=0,
			xmax = 100,
			xtick = {0,20,..., 150},
			ymin=0,
			ymax=4.7e-2,
			yticklabels = {,,},
			scaled y ticks = false,
			xlabel = {$\| \ma \Gamma^{1/2}(\tve \theta - \hve\theta)\|_2^2$},
			ylabel = probabiliy density,
			]
			\nextgroupplot[title = {$n=256$}, xlabel = { }, 
				ylabel = {\hspace{18pt} Projected NLS \newline probability density}
			]

			\addplot[black, thick]  table [x=x,y=y]{\kde};	
			\addplot[black, thick, densely dotted]
				 table [x=x,y=y]{fig_nmr_pert_256_full.dat};
			\addplot[fill=colorbrewerA1, draw=none, opacity=0.5]
				 table [x=x,y=y]{fig_nmr_pert_256_projected.dat};
			
			\nextgroupplot[title = {$n=1024$}, xlabel = { }]
			\addplot[black, thick]  table [x=x,y=y]{\kde};	
			\addplot[black, thick, densely dotted]
				 table [x=x,y=y]{fig_nmr_pert_1024_full.dat};
			\addplot[fill=colorbrewerA1, draw=none, opacity=0.5]
				 table [x=x,y=y]{fig_nmr_pert_1024_projected.dat};
			
			\nextgroupplot[title = {$n=4096$}, xlabel = { }]
			\addplot[black, thick]  table [x=x,y=y]{\kde};	
			\addplot[black, thick, densely dotted]
				 table [x=x,y=y]{fig_nmr_pert_4096_full.dat};
			\addplot[fill=colorbrewerA1, draw=none, opacity=0.5]
				 table [x=x,y=y]{fig_nmr_pert_4096_projected.dat};

			% Second row
			\nextgroupplot[ylabel = {\hspace{26pt} Fast HSVD \newline probability density}]
			\addplot[black, thick]  table [x=x,y=y]{\kde};	
			\addplot[black, thick, densely dotted]
				 table [x=x,y=y]{fig_nmr_pert_256_full.dat};
			\addplot[fill=colorbrewerA4, draw=none, opacity=0.5] table
				 [x=x,y=y]{fig_nmr_pert_256_hsvd_fast.dat};
			
			\nextgroupplot[]
			\addplot[black, thick]  table [x=x,y=y]{\kde};	
			\addplot[black, thick, densely dotted]
				 table [x=x,y=y]{fig_nmr_pert_1024_full.dat};
			\addplot[fill=colorbrewerA4, draw=none, opacity=0.5] table
				 [x=x,y=y]{fig_nmr_pert_1024_hsvd_fast.dat};
			
			\nextgroupplot[]
			\addplot[black, thick]  table [x=x,y=y]{\kde};	
			\addplot[black, thick, densely dotted]
				 table [x=x,y=y]{fig_nmr_pert_4096_full.dat};
			\addplot[fill=colorbrewerA4, draw=none, opacity=0.5] table
				 [x=x,y=y]{fig_nmr_pert_4096_hsvd_fast.dat};
		\end{groupplot}
	\end{tikzpicture}
	\caption{
		The density of standardized error in the parameter estimate $\tve\theta$, 
		$\ma \Gamma^{1/2}(\tve \theta - \hve\theta)$ as described in \cref{sec:example:precision}.
		Thick curves show the expected distribution of the 2-norm of the standardized error
		and the filled regions show the empirically determined density from $4000$ realizations 
		of each method for each $n$.
		The dotted lines shows the density of the standardized error
		in the full nonlinear least squares parameter estimate.
		}
	\label{fig:nmr_pert}
\end{figure}
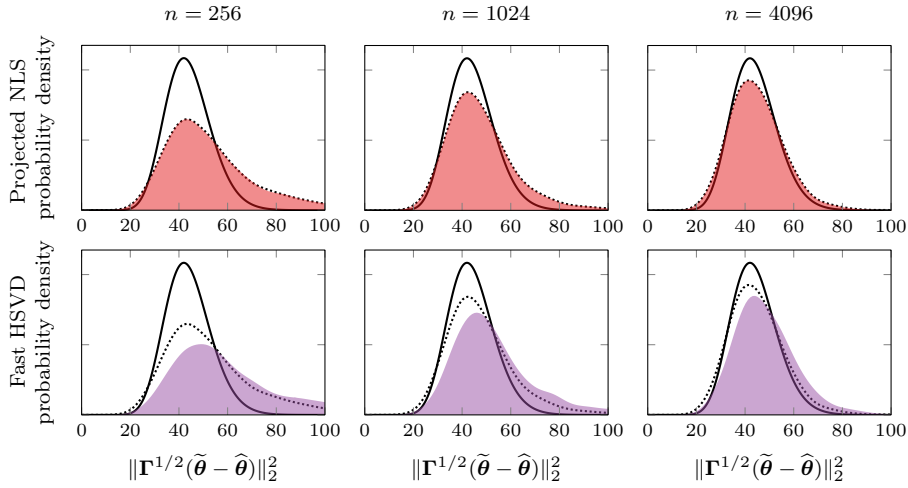

\section{Discussion}
In this paper we have shown that by solving a sequence of projected nonlinear least squares problems
we can substantially improve the run time performance with a negligible loss of accuracy
for solving exponential fitting problem 
when compared to both conventional nonlinear least squares and 
HSVD, a typical subspace based approach.
For the exponential fitting problem, 
there are still several open questions.
Is there a better choice for selecting interpolation points than our box partition?
Are there better subspaces, such as one that includes not only $\ma V(\ve\mu)$, but $\ma V'(\ve\mu)$ as well?
More generally: can we provide conditions that guarantee the convergence
of the series of projected problems?
Can we bound the error incurred by the projection in a deterministic sense?
Finally, we ask, what were the key features that allowed the projected approach to 
work for the exponential fitting problem and could this approach be applied to other problems?
One key feature was that projected model $\ma W^*_\ell \ve f(\ve\theta)$ and 
Jacobian $\ma W_\ell^*\ma J(\ve\theta)$ could be computed inexpensively.
The other key feature was that we were able to generate subspaces $\set W_\ell$
such that the subspace angles between $\set W_\ell$ and the range of the Jacobian $\set J(\ve\theta)$
remained small.
These two requirements limit the applicability of these results to specific pairs of models $\ve f(\ve\theta)$ and
subspaces $\set W_\ell$, but for those problems that satisfy these requirements
the projected nonlinear least squares approach presents a way to improve performance.

\section*{Acknowledgements}
I would like to thank Mark Embree, Caleb Magruder, and Paul Constantine for their help in preparing this manuscript;
Christopher Beattie, the Einstein Foundation of Berlin, Volker Mehrmann, and TU Berlin 
for hosting me during the writing of this manuscript;
and my PhD committee: Steve Cox, Thanos Antoulas, and Matthias Heinkenschloss.

\appendix
\section{Projected least squares error bounds\label{sec:bounds}}
Here we provide two lemmas used in \cref{sec:optimization}
related to the accuracy of a projected least squares problem.
 
\begin{lemma}\label{lem:project_ls}
	Let $\ma A\in \C^{n\times q}$ have full column rank and let $\ve b\in \C^n$
	with respective range $\set A$ and span $\set B$.
	Let $\set W$  be an $m$-dimensional subspace of $\C^n$ where $m\ge q$ 
	and let $\ma P_\set W$ be the orthogonal projector onto $\set W$ where $\ma P_\set W\ma A$ has full column rank.
	If $\ve x$ is the minimizer of $\| \ma A\ve x - \ve b\|_2$
	and $\ve y$ is the minimizer of $\| \ma P_\set W (\ma A\ve y - \ve b)\|_2$
	then
	\begin{equation}
		\| \ve x \!-\! \ve y\|_2 \le \!
		\| \ma A^+\|_2 \| \ve b\|_2 \!\left[
			\sin \phi_q(\set W, \set A)	\sin \phi_1(\set W, \set B) 
			\!+\! \tan^2 \phi_q(\set W, \set A) \cos \phi_1(\set W, \set B)
		\right]\!.\!\!\!
	\end{equation}
\end{lemma}
\begin{proof}
	Using the pseudoinverse, we write $\ve x$ and $\ve y$ as
	\begin{align}
		\ve x &= (\ma A^*\ma A)^{-1}\ma A^*\ve b, \\
		\ve y &= (\ma A^*\ma P_\set W\ma A)^{-1} \ma A^*\ma P_\set W\ve b.
	\end{align}
	Inserting the decomposition of the identity $\ma I = \ma P_\set W + \ma P_\set W^\perp$ before $\ve b$ in $\ve x$,
	\begin{equation}
		\ve x = (\ma A^*\ma A)^{-1}\ma A^*\ma P_\set W\ve b + (\ma A^*\ma A)^{-1}\ma A^*\ma P_\set W^\perp\ve b,
	\end{equation}
	we then note the difference between $\ve x$ and $\ve y$ is
	\begin{equation}
		\ve x - \ve y = (\ma A^*\ma A)^{-1}\ma A^*\ma P_\set W^\perp \ve b + 
				\left[(\ma A^*\ma A)^{-1} - (\ma A^*\ma P_\set W\ma A)^{-1}\right]\ma A^*\ma P_\set W\ve b.
	\end{equation}
	Replacing $\ma A$ with its short form SVD, $\ma A = \ma U\ma \Sigma \ma V^*$,
	and $\ma P_\set W$ with $\ma W\ma W^*$ where $\ma W$ is an orthonormal basis for $\set W$, 
	we have
	\begin{align*}
		\ve x - \ve y 
		&= \ma V\ma \Sigma^{-1}\ma U^*\ma P_\set W^\perp \ve b
			- \ma V\ma \Sigma^{-1}\left[
				(\ma U^*\ma W\ma W^*\ma U)^{-1}
				-\ma I
			  \right]\ma \Sigma^{-1}\ma U^*\ma W \ma W^* \ve b, \\
		\|\ve x \!-\! \ve y\|_2 &
		\le \|\ma \Sigma^{-1}\|_2 \left( 
			\| \ma P_\set W^\perp \ma U\|_2 \| \ma P_\set W^\perp \ve b\|_2 
			\!+\! \| (\ma U^*\ma W\ma W^*\ma U)^{-1} \! - \ma I\|_2 \| \ma U^*\ma W\|_2 \| \ma W^*\ve b\|_2
			\right)\!.\!\!\!\!
	\end{align*}
	Then invoking the subspace angle identities:
	$\| \ma P_\set W^\perp\ve b\|_2 = \sin \phi_1(\set W, \set B) \|\ve b\|_2$,
	$\| \ma W^*\ve b\|_2 \!=\! \cos\phi_1(\set W, \set B)\|\ve b\|_2$,
	$\| \ma U^*\ma W\|_2 = \cos \phi_1(\set W, \set A)$,
	and $\|\ma P_\set W^\perp \ma U\|_2 = \sin \phi_q(\set A, \set W)$,
	\begin{multline}
		\| \ve x - \ve y\|_2 \le \|\ma \Sigma^{-1}\|_2 \| \ve b \|_2 \big(  
			\sin \phi_q(\set W, \set A) \sin \phi_1(\set W, \set B) \\
			+ \| (\ma U^*\ma W\ma W^*\ma U)^{-1} - \ma I\|_2 \cos \phi_1(\set W, \set A)\cos \phi_1(\set W, \set B)
			\big).
	\end{multline}
	To bound $\| (\ma U^*\ma W\ma W^*\ma U)^{-1} - \ma I\|_2$,
	we note that as $\ma U^*\ma W\ma W^*\ma U$ is positive semidefinite, 
	there exists an $\alpha \ge 0$ such that $\ma U^*\ma W\ma W^*\ma U \succeq \alpha^2\ma I$.
	This implies
	\begin{align}
		\lambda_k(\ma U^*\ma W\ma W^*\ma U) - \alpha^2  \ge 0,
		\quad \Rightarrow \quad
		\sigma_k(\ma W^*\ma U)^2 - \alpha^2  \ge 0, \quad \forall k \in 1, \ldots, q
	\end{align}
	where $\lambda_k(\ma X)$ is the $k$ eigenvalue in descending order of $\ma X$.
	The largest $\alpha$ satisfying this inequality is 
	$\alpha = \sigma_q(\ma W^*\ma U) = \cos \phi_q(\set W, \set A)$.
	Invoking~\cite[Cor.~7.7.4]{HJ85}, 
	$\ma U^*\ma W\ma W^*\ma U \succeq \alpha^2 \ma I$
	implies  $(\ma U^*\ma W\ma W^*\ma U)^{-1} \preceq \alpha^{-2}\ma I$
	and hence
	\begin{align}
		(\ma U^*\ma W\ma W^*\ma U)^{-1} - \ma I &\preceq \alpha^{-2}\ma I - \ma I = (\alpha^{-2} - 1) \ma I.
	\end{align}
	Upon taking the norm, we have 
	\begin{align}
		\| (\ma U^*\ma W\ma W^* \ma U)^{-1} - \ma I \|_2 &\le (\alpha^{-2} - 1).
	\end{align}
	Thus $\alpha^{-2} - 1 = \sec^2\phi_q(\set W, \set A)$
	and invoking trigonometric identities, $\sec^2 \phi_q(\set W, \set A) - 1 = \tan^2 \phi_q(\set W, \set A)$;
	hence
	\begin{multline}
		\| \ve x - \ve y\|_2 \le \| \ma \Sigma^{-1}\|_2 \|\ve b\|_2 \big(
			\sin \phi_q(\set W, \set A) \sin \phi_1(\set W, \set B) \\
			+ \tan^2 \phi_q(\set W, \set A) \cos \phi_1(\set W, \set A) \cos \phi_1(\set W,  \set B)
		\big). 
	\end{multline}
	By applying the upper bound $\cos \phi_1(\set W, \set A) \le 1$, 
	and noting $\|\ma \Sigma^{-1}\|_2 = \|\ma A^{+}\|_2$ we obtain the desired bound.
\end{proof}

\begin{lemma}\label{lem:pls_normal}
	In the same setting as \cref{lem:project_ls},
	\begin{align}
		\| \ma A^*\ma A \ve y - \ma A^*\ve b\|_2 &\le 
		 \frac{\cos \phi_1(\set A, \set W) \sin \phi_q(\set A, \set W)}{\cos^2\phi_q(\set A, \set W)} 
			\| \ma A\|_2\| \ma P_\set A^\perp \ve b\|_2.
	\end{align}
\end{lemma}
\begin{proof}
	Using the pseudoinverse, 
	\begin{align}
		\ma A^*\ma A \ve y - \ma A^*\ve b =
			\ma A^*\ma A(\ma A^*\ma P_\set W\ma A)^{-1} \ma A^*\ma P_\set W \ve b - \ma A^*\ve b. 
	\end{align}
	Then, inserting the decomposition of the identity $\ma I = \ma P_\set A + \ma P_\set A^\perp$
	between $\ma P_\set W$ and $\ve b$, 
	\begin{align}
	\ma A^*\ma A\ve y - \ve b =& \,
		 \ma A^*\ma A(\ma A^*\ma P_\set W\ma A)^{-1} \ma A^*\ma P_\set W(\ma P_\set A + \ma P_\set A^\perp) \ve b
				 - \ma A^*\ve b.
	\end{align}
	After expanding the first term on the right, the $\ma P_\set A$ component is $\ma A^*\ve b$, 
	\begin{align}
		 \ma A^{\!*}\ma A(\ma A^{\!*}\ma P_\set W\ma A)^{-1}\! \ma A^{\!*}\ma P_\set W\ma P_\set A\ve b
		\!=\! \ma A^{\!*}\ma A(\ma A^{\!*}\ma P_\set W\ma A)^{-1}\! \ma A^{\!*}\ma P_\set W \ma A (\ma A^{\!*}\ma A)^{-1}\! \ma A^{\!*} \ve b
		\!=\! \ma A^{\!*}\ve b,
	\end{align}
	and hence cancels $\ma A^*\ve b$ leaving one term:
	\begin{align}
		\ma A^*\ma A\ve y - \ma A^*\ve b = 
		 \ma A^*\ma A(\ma A^*\ma P_\set W\ma A)^{-1} \ma A^*\ma P_\set W \ma P_\set A^\perp \ve b.
	\end{align}
	Next, we define the oblique projector above 
	$\ma X := \ma A(\ma A^*\ma P_\set W\ma A)^{-1}\ma A^*\ma P_\set W$
	in terms of the SVD of $\ma A$.
	If $\ma A$ has a full and reduced SVD
	\begin{equation}
		\ma A = \ma U \ma \Sigma \ma V^* = \begin{bmatrix} \ma U_1 & \ma U_2 \end{bmatrix}
				\begin{bmatrix} \ma \Sigma_1 \\ \ma 0\end{bmatrix}
				\ma V^*
			= \ma U_1 \ma \Sigma_1 \ma V^*,
	\end{equation}
	then this oblique projector is
	\begin{align*}
		\ma X &= \ma U\ma \Sigma \ma V^*(\ma V^*\ma \Sigma^* \ma U\ma W\ma W^*\ma U\ma \Sigma \ma V^*)^{-1} \ma V\ma \Sigma^*\ma U^*\ma W\ma W^*\\
			&= \ma U_1\ma \Sigma_1 \ma V^*\ma V(\ma \Sigma_1 \ma U_1^*\ma W\ma W^*\ma U_1\ma \Sigma_1 )^{-1}\ma V^* \ma V\ma \Sigma_1^*\ma U_1^*\ma W\ma W^*\\
			&= \ma U_1\ma \Sigma_1 \ma \Sigma_1^{-1}(\ma U_1^*\ma W\ma W^*\ma U_1)^{-1}\ma \Sigma_1^{-1}\ma \Sigma_1^*\ma U_1^*\ma W\ma W^*\\
			&= \ma U_1(\ma U_1^*\ma W\ma W^*\ma U_1)^{-1}\ma U_1^*\ma W\ma W^*.
	\end{align*}
	Inserting this result into the expression for $\ma A^*\ma A\ve y - \ma A^*\ve b$, we obtain the bound	
	\begin{align*}
		\|\ma A^*\ma A \ve y - \ma A^*\ve b \|_2 &= \|\ma A^* \ma U_1(\ma U_1^*\ma W\ma W^*\ma U_1)^{-1}\ma U_1^*\ma W\ma W^*\ma U_2\ma U_2^*\ve b\|_2 \\
		&\le \|\ma A\|_2 \|(\ma U_1^*\ma W\ma W^*\ma U_1)^{-1}\|_2 \|\ma U_1^*\ma W\|_2 \|\ma W^*\ma U_2\|_2 \| \ma U_2^*\ve b\|_2 \\
		&= \sigma_q(\ma U_1^*\ma W)^{-2} \sigma_1(\ma U_1^*\ma W) \sigma_1(\ma W^*\ma U_2) \| \ma A\|_2\| \ma P_\set A^\perp \ve b\|_2\\ 
		&= \frac{\cos \phi_1(\set A, \set W) \sin \phi_q(\set A, \set W)}{\cos^2\phi_q(\set A, \set W)} \| \ma A\|_2\| \ma P_\set A^\perp \ve b\|_2. 
	\end{align*}
\end{proof}

\section{Generalized geometric sum formula\label{sec:geometric}}
A critical component for our algorithm is the ability to compute in closed form
\emph{generalized geometric sum},
\begin{equation}\label{eq:ggsum}
	\sum_{k=n_1}^{n_2-1} k^p e^{\delta k} \quad 
\end{equation}
where $\delta \in \C$, and  $p$, $n_1$, $n_2$ are non-negative integers.
The standard geometric sum formula provides a closed form expression when $p=0$
and when $e^\delta=1$, this sum is can be written in terms of Bernoulli polynomials~\cite[eq.~(24.4.9)]{DLMF}.
The following lemma establishes the remaining case when $e^\delta \ne 1$ and $p > 0$.
\begin{lemma}
	Let $\delta\in \C$ with $e^\delta\ne 1$, $p,n_1,n_2\in \Z_+$, where $0\le n_1\le n_2$,
	then 
	\begin{equation}\label{eq:generalized_geometric}
		\sum_{k=n_1}^{n_2 - 1} k^p e^{\delta k} = 
			\sum_{\ell=0}^{p} \frac{ \chi_{n_1} (p, \ell) e^{\delta (n_1+\ell)} - \chi_{n_2} (p,\ell) e^{\delta (n_2 +\ell} }
				{ (1-e^{\delta})^{\ell+1}}
	\end{equation}
	where $\chi_n(p,\ell)$ is given by the recurrence
	\begin{equation}\label{eq:chi}
		\chi_n(p+1,\ell) = (n+\ell)\chi_n(p,\ell) + k \, \chi_n(p,\ell-1); \qquad
		\chi_n(0,\ell) = \delta_{\ell,0}, 
		 \quad p,\ell \ge 0.
	\end{equation}
\end{lemma}
\begin{proof}
	Multiplying each term of the geometric sum by $k^p$ corresponds to a $p$th derivative with respect to $\delta$
	of each entry.
	Since this is a finite sum, we pull the derivative outside the sum yielding
	\begin{equation}\label{eq:dersum}
		\sum_{k=n_1}^{n_2-1} k^p e^{\delta k} 
			= \sum_{k=n_1}^{n_2-1} \frac{\partial^p}{\partial \delta^p}e^{\delta k}
			= \frac{\partial^p}{\partial \delta^p} \sum_{k=n_1}^{n_2-1} e^{\delta k}
			= \frac{\partial^p}{\partial \delta^p} \frac{e^{\delta n_1} - e^{\delta n_2}}{1-e^\delta}.
	\end{equation}
	To obtain an explicit formula for the derivative on the right,
	we show by induction	
	\begin{equation}\label{eq:recurrence}
		\frac{\partial^p}{\partial \delta^p} \frac{e^{n\delta}}{1-e^\delta} 
			= \sum_{\ell=0}^p \chi_n(p,\ell) \frac{e^{(n+\ell)\delta}}{(1-e^\delta)^{\ell+1}}.
	\end{equation}
	The base case $p=0$ holds as $\chi_n(0,0) = 1$.
	The inductive step follows by taking the derivative of each side
	\begin{align*}
	\frac{\partial}{\partial \delta} \sum_{\ell=0}^p \chi_n(p,\ell) \frac{e^{(n+\ell)\delta}}{(1-e^\delta)^{\ell+1}}% 
	& = \sum_{\ell=0}^{p+1} 
		\left[ \chi_n(p,\ell) (n+\ell) + \chi_n(p,\ell-1) \ell \right] 
		\frac{e^{(n+\ell)\delta}}{(1-e^\delta)^{\ell+1}} \\ 
	& = \sum_{\ell=0}^{p+1} \chi_n(p+1,\ell) \frac{e^{(n+\ell)\delta}}{(1-e^\delta)^{\ell+1}}.
	\end{align*}
	Subtracting~\cref{eq:recurrence} evaluated at $n = n_2$ from \cref{eq:recurrence} evaluated at $n=n_1$
	yields \cref{eq:generalized_geometric}.
\end{proof}

With this lemma, we now state the generalized geometric sum formula.

\begin{theorem}[Generalized Geometric Sum Formula]\label{thm:polyexp_sum}
	Given $\delta\in \C$, $p\in \Z_+$, and $n_1,n_2\in \Z$ with $0\le n_1< n_2$ integers, then 
	\begin{equation}\label{eq:polyexp_sum}
	\sum_{k=n_1}^{n_2-1} k^p e^{\delta k} = \begin{cases}
	\displaystyle \sum_{\ell=0}^p \frac{\chi_{n_1}(p,\ell)e^{(n_1+\ell)\delta} 
		- \chi_{n_2}(p,\ell)e^{ (n_2+\ell)\delta}}{(1-e^\delta)^{\ell+1}},
	& e^\delta \ne 1; \vspace{2pt}\\
	\displaystyle 
	\phantom{\sum_{k=0}^p} \frac{B_{p+1}(n_2) - B_{p+1}(n_1)}{p+1}, & e^\delta = 1;
	\end{cases}
	\end{equation}
	where $B_p$ is the $p$th Bernoulli polynomial.
\end{theorem}

\bibliographystyle{siamplain}
\bibliography{abbrevjournals,master}

\end{document}